\documentclass[12pt,twoside]{amsart}
\usepackage[margin=3cm]{geometry}
\usepackage[colorlinks=false]{hyperref}
\usepackage[english]{babel}
\usepackage{graphicx,titling}
\usepackage{float}
\usepackage{amsmath,amsfonts,amssymb,amsthm}
\usepackage{lipsum}
\usepackage[T1]{fontenc}
\usepackage{fourier}
\usepackage{color}
\usepackage[latin1]{inputenc}
\usepackage{esint}
\usepackage{caption}
\usepackage{ccicons}

\makeatletter
\def\blfootnote{\xdef\@thefnmark{}\@footnotetext}
\makeatother

\newcommand\ccnote{
    \blfootnote{\ccLogo\, \ccAttribution\,\, Licensed under a Creative Commons Attribution License (CC-BY).}
}

\usepackage[export]{adjustbox}
\numberwithin{equation}{section}
\usepackage{setspace}
\renewcommand{\le}{\leqslant}

\renewcommand{\ge}{\geqslant}

\renewcommand{\mathbb}{\varmathbb}
\usepackage{fancyhdr}
\newtheorem{theorem}{Theorem}[section]
\newtheorem{lemma}[theorem]{Lemma}
\newtheorem{corollary}[theorem]{Corollary}
\newtheorem{proposition}[theorem]{Proposition}
\newtheorem{definition}[theorem]{Definition}
\newtheorem{remark}[theorem]{Remark}
\newtheorem{claim}[theorem]{Claim}

\usepackage{comment}
\usepackage{enumerate}
\usepackage{hyperref}
\usepackage[alphabetic, msc-links, backrefs]{amsrefs}

\newcommand{\Ff}{\mathcal F}

 \newcommand{\Dd}{\mathcal{D}}

  \newcommand{\MM}{\mathcal{M}}
  
\newcommand{\Pp}{\mathcal{P}}

 \newcommand{\RR}{\mathbf{R}}  % reals
 \newcommand{\ZZ}{\mathbf{Z}}  % integers
 \newcommand{\BB}{\mathbf{B}}  %ball
  \newcommand{\Div}{\operatorname{Div}}
  
    \newcommand{\dist}{\operatorname{dist}}
 
 \newcommand{\area}{\operatorname{area}}
 \newcommand{\eps}{\epsilon}
 \newcommand{\Tan}{\operatorname{Tan}}

 \newcommand{\NN}{\mathbf{N}}

\newcommand{\spt}{\operatorname{spt}}
\newcommand{\Hh}{\mathcal{H}}
\usepackage{amsthm}

\newcommand{\pdf}[2]{\frac{\partial #1}{\partial #2}}

\input epsf
\def\begfig {
\begin{figure}
\small }
\def\endfig {
\normalsize
\end{figure}
}
\newcommand{\Xx}{\mathcal{X}}
\newcommand{\tY}{\tilde Y}

\address{Brian White, Stanford University, Department of Mathematics, 450 Jane Stanford Way, Bldg 380}
\email{bcwhite@stanford.edu}

\begin{document}

\thispagestyle{empty}

\begin{minipage}{0.28\textwidth}
\begin{figure}[H]
%\centering
\includegraphics[width=2.5cm,height=2.5cm,left]{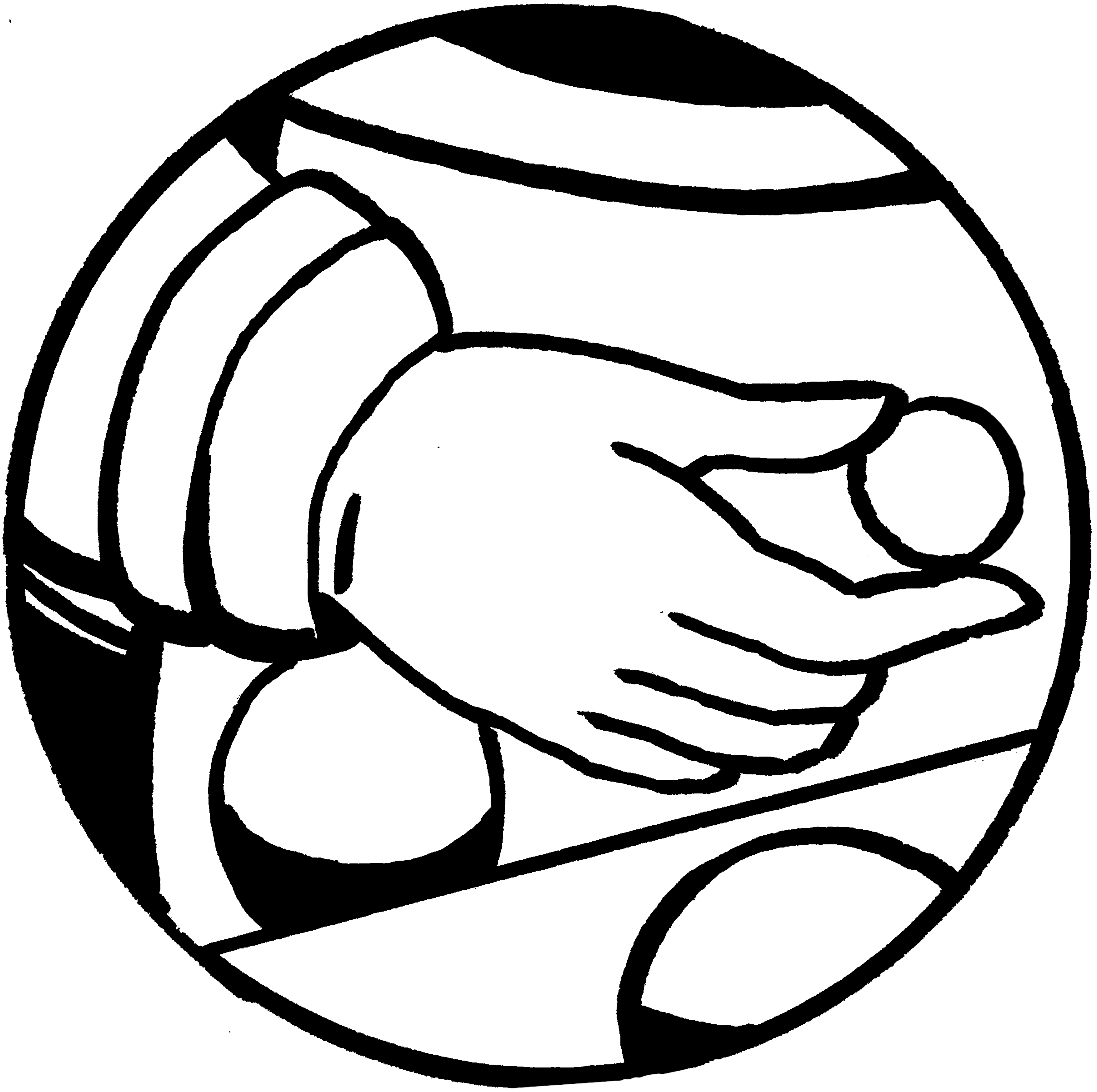}
\end{figure}
\end{minipage}
\begin{minipage}{0.7\textwidth} 
\begin{flushright}
%% The following metadata, in particular
%% the Paper No. and the DOI will be inserted by the journal
Ars Inveniendi Analytica (2021), Paper No. 4, 43 pp.
\\
DOI 10.15781/vks5-5e33
\end{flushright}
\end{minipage}

\ccnote

\vspace{1cm}

%%      -------------------------------------------------------------------------------
%%      -------------------------- TITLE ----------------------------
%%      -------------------------------------------------------------------------------
%% Authors, please put here the full title of the article

\begin{center}
\begin{huge}
\textit{Mean Curvature Flow with Boundary}

%\textit{some titles take two lines}

\end{huge}
\end{center}

\vspace{1cm}

%%      -------------------------------------------------------------------------------
%%      -------------------------- AUTHORS AND AFFILIATIONS ----------------------------
%%      -------------------------------------------------------------------------------
%% Authors, please put here your full names and affiliations

\begin{center}
\begin{minipage}[t]{.28\textwidth}
\begin{center}
{\large{\bf{Brian White}}} \\
\vskip0.15cm
\footnotesize{Stanford University}
\end{center}
\end{minipage}
\end{center}

\vspace{1cm}

%%% Please replace "James Mustard" below 
%%% with the name of the managing editor for your submission.
%%% If you are unsure about their identity
%%% please ask an editor-in-chief about.

\begin{center}
\noindent \em{Communicated by Carlo Sinestrari}
\end{center}
\vspace{1cm}

%%      -------------------------------------------------------------------------------
%%      -------------------------- BEGIN ABSTRACT ----------------------------
%%      -------------------------------------------------------------------------------
%% Authors, please put here the ABSTRACT and KEYBOARDS

\noindent \textbf{Abstract.} \textit{We develop a theory of surfaces with boundary moving by mean curvature flow.
In particular, we prove a general existence theorem using elliptic regularization,
 and we prove boundary regularity
at all positive times under very mild hypotheses.}
\vskip0.3cm

\vspace{0.5cm}

%%      -------------------------------------------------------------------------------
%%      -------------------------- BEGIN ARTICLE ----------------------------
%%      -------------------------------------------------------------------------------
%% Authors, copy the body of your paper here

\tableofcontents

\section{Introduction}

In this paper, we study mean curvature flow for surfaces with boundary:
each point moves so that the normal component of its velocity is equal to the mean
curvature, and the boundary remains fixed.  (More generally, the boundary can
be time-dependent, but prescribed.)
In particular,
\begin{enumerate}
\item\label{general-item-intro} We define  {\bf integral Brakke flows with boundary} and prove the basic properties.
       This is a rather general class that includes network flows.   See~\S\ref{brakke-flow-section}.
\item We define the subclass of {\bf standard Brakke flows with boundary}.
   These are flows as in~\eqref{general-item-intro} with additional nice properties.  In particular, for almost all
   times, the moving surface has the prescribed boundary in the sense
   of mod $2$ homology.   This condition excludes, for example, surfaces
   with triple junctions (or, more generally, with odd-order junctions).  See~\S\ref{standard-section}. 
\item Following Ilmanen~\cite{ilmanen-elliptic}, we use elliptic regularization to prove existence
   of standard Brakke flows with boundary for any prescribed initial surface. See~\S\ref{existence-section}.
\item We prove a strong boundary regularity theorem for standard Brakke flows
   with boundary.  See~\S\ref{boundary-regularity-section}.
\end{enumerate}

As a special case of some of the results
 (Theorems~\ref{existence-theorem} and~\ref{main-boundary-regularity-theorem}), we have

\begin{theorem}\label{intro-theorem}
Let $N$ be a smooth, compact, $(m+1)$-dimensional Riemannian manifold with smooth, strictly mean-convex boundary.
Let $M_0$ be a smoothly embedded $m$-dimensional submanifold of $N$ whose boundary is 
a smooth submanifold $\Gamma$ of $\partial N$.
(More generally, $M_0$ can be any $m$-rectifiable set of finite $m$-dimensional measure whose boundary,
in the sense of mod $2$ flat chains, is a smoothly embedded submanifold $\Gamma$ of $\partial N$.)
Then there is a standard Brakke flow 
\[
    t\in [0,\infty) \mapsto M(t)
\]
with boundary $\Gamma$ such that $M(0)=M_0$.
Furthermore, if $M(\cdot)$ is any  standard Brakke flow with boundary $\Gamma$, then the flow is smooth
(with multiplicity one) in a spacetime neighborhood of each point $(p,t)$ with $p\in \Gamma$ and $t>0$.
\end{theorem}

Thus (under the hypotheses of the theorem) we have boundary regularity at {\bf all} positive times, even after
interior singularities may have occurred.

The regularity in Theorem~\ref{intro-theorem} is  uniform as $t\to\infty$. 
For given any sequence of times $t_i\to\infty$, there is a subsequence $t_{i(j)}$ such that the time-translated
flows 
\[
  t\in [-t_{i(j)}, \infty) \mapsto M_{i(j)}(t):= M(t_{i(j)} + t)
\]
 converge to a standard eternal limit flow $M'(\cdot)$
by~\S\ref{compactness-section} and Theorem~\ref{closure-theorem}.  Since the area of $M(t)$ is a decreasing function of $t$,
the area of $M'(t)$ is constant (it is equal to the limit as $t\to\infty$ of the area of $M(t)$).  It follows that the $M'(t)$ are stationary integral varifolds
and therefore non-moving (i.e., independent of $t$).
The limit flow is regular
at the boundary by Theorem~\ref{main-boundary-regularity-theorem},
 and thus the convergence $M_{i(j)}(\cdot)\to M'(\cdot)$ is smooth
near the boundary by the local regularity theory in~\cite{white-local}.

The notion of {\bf standard} Brakke flow with boundary is crucial in Theorem~\ref{intro-theorem}:
the regularity assertion of Theorem~\ref{intro-theorem} is false for general integral Brakke flows
with boundary, because interior singularities can move into  the boundary.
Consider, for example, three points $A$, $B$, and $C$ on the unit circle in $\RR^2$
and consider a configuration consisting of three curves in the interior of the triangle $ABC$ 
such that the three curves meet at equal angles at a point $P$ in the interior of the disk and such 
the other endpoints of the curves are the three points $A$, $B$, and $C$. The configuration evolves so that the three points on the unit circle are fixed,
and so that interior points move with normal velocity equal to the curvature.  This implies
that the triple junction $P(t)$ moves in such a way that the curves continue to meet at equal
angles at the junction.  If each interior angle of the triangle $ABC$ is less than $120^\circ$,
then the triple junction remains in the interior, and we have boundary regularity at all times.
However, if one of the angles is greater than $120^\circ$, then $P(t)$ will bump into the corresponding
vertex in finite time, thus creating a boundary singularity.

The flow described in the previous paragraph is an integral Brakke flow with boundary $\{A,B,C\}$.
However, it is not a standard Brakke flow with boundary $\{A,B,C\}$, because if we think
of the network as a mod $2$ chain, then the boundary contains $P(t)$ in addition to $A$, $B$, and $C$.

It is natural to wonder whether such a boundary singularity could occur if the original surface
is smooth and embedded.  In the case of curves, the answer is ``no": the flow would remain smooth everywhere for all time
by the analog of Grayson's Theorem.  However, although I do not yet have a proof, I believe
that there is an integral Brakke flow 
\[
   t\in [0,\infty)\mapsto M(t)
\]
with boundary $\Gamma$, where $\Gamma$ consists of smooth embedded curves in the unit
sphere in $\RR^3$, such that $M(0)$ is a smoothly embedded surface in the unit ball
and such that later the moving surface develops  a triple junction curve that eventually
bumps into the boundary.
Note that this could only happen if we had non-uniqueness, since by 
Theorem~\ref{intro-theorem} there is a standard Brakke flow $M'(\cdot)$ with the same
initial surface and the same boundary, and that flow never develops boundary singularities.
Of course the two flows are equal at least until singularities occur,  but they must differ as soon as $M(\cdot)$
has a triple junction curve.

In Theorem~\ref{intro-theorem}, the condition that $\Gamma$ lie in the boundary of $N$
is also crucial.  In another paper~\cite{white-singularity}, we show that there is a standard mean curvature flow with boundary that starts 
with a smoothly 
embedded M\"obius strip in $\RR^3$ and that develops a boundary singularity
at which the tangent flow is given by a smoothly embedded, non-orientable shrinker with straight line boundary. 
For oriented surfaces, the situation is very different: we prove~\cite{white-singularity}*{Theorem~1} that
if $M$ is an $m$-dimensional, smoothly embedded shrinker in $\RR^{m+1}$ with an $(m-1)$-dimensional
linear subspace as boundary, then $M$ is a flat halfspace.

The regularity part of Theorem~\ref{intro-theorem} is a consequence of the
following general theorem (see Theorem~\ref{general-boundary-regularity-theorem}):

\begin{theorem}\label{intro-theorem-2}
Suppose that $t\in [0,T]\mapsto M(t)$ is an $m$-dimensional standard mean curvature
flow with boundary $\Gamma$ in a smooth, $(m+1)$-dimensional Riemannian manifold.
If a tangent flow at $(p,t)$ is contained in a wedge, where $p\in \Gamma$ and $t>0$, 
then $(p,t)$ is a regular point of the flow $M(\cdot)$.
\end{theorem}

Two features of this paper seem to be new even for Brakke flows without boundary.
First, when taking limits of Brakke flows, we get  improved subsequential convergence
of the mean curvature for almost all times; see Remark~\ref{new-h-remark}.
Second, 
to prove Huisken's monotonicity formula, one needs to know that the mean curvature vector is orthogonal
to the variety almost everywhere.  Brakke~\cite{brakke}*{\S5} proved such orthogonality
for arbitrary integral varifolds of bounded first variation.  However, the proof is rather
long (40 pages).  In this paper, we give a much easier proof that such orthogonality is preserved when
taking weak limits of mean curvature flows. Thus, in particular, orthogonality holds in 
flows coming from elliptic regularization. 
For this reason, we have chosen to include orthogonality of mean curvature as part of the definition
of Brakke flow.

For simplicity, in most of the paper we consider flows in which the boundary is fixed.
  In \S\ref{moving-boundary-section}, we indicate 
how to modify the theory for moving boundaries.

In~\S\ref{orientation-section}, we show that a certain weak notion of orientability 
is preserved when taking weak limits of flows.

Although mean curvature flow has been extensively studied, there have been
only a few investigations of mean curvature flow of surfaces with boundary.
The papers \cite{white-topology} and \cite{white-local} dealt with mean curvature flow of surfaces
both with and without boundary.
In \cite{stone}, Stone proved a theorem analogous to the boundary regularity part 
of Theorem~\ref{intro-theorem}, but only
at the first singular time and under additional, rather restrictive hypotheses.  In particular, the moving surface
was assumed to be mean convex and to satisfy a Type I estimate.
In~\cite{ilmanen-white-cones}, mean curvature flow with boundary was used to prove
sharp lower density bounds for area-minimizing hypercones.

\section{Notation}\label{notation-section}

In this paper, $U$ is a smooth Riemannian manifold (possibly with smooth boundary).  
We do not assume that $U$ is complete: it may be an open subset
of a larger Riemannian manifold.  We let $G_m(U)$ denote the Grassman bundle of 
pairs $(x,P)$ where $x\in U$ and $P$ is an $m$-dimensional linear subspace of $\Tan(U,x)$.
We let $\Xx(U)$ denote the space of continuous, compactly supported vectorfields on $U$.
We let $\Xx_m(U)$ denote the space of continuous, compactly supported functions on $G_m(U)$
that assign to each $(x,P)$ in $G_m(U)$ a vector in $\Tan(U, x)$.

If $M$ is a Radon Measure on $U$ and if $f$ is a function on $U$, we let
\[
   Mf = \int f\,dM.
\]

If $\Gamma$ is a $k$-dimensional submanifold of $U$ (or, more generally, a $k$-rectifiable
set of locally finite $k$-dimensional measure), then (by slight abuse of notation) we
will also use $\Gamma$ to denote the associated Radon measure.  
Thus
\[
  \int f\,d\Gamma = \int_\Gamma f\,d\Hh^k
\]
and
\[
  \Gamma(K) = \Hh^k(\Gamma\cap K).
\]

\section{$L^p$ vectorfields}

\begin{comment}
\begin{definition}
Let $\Xx_m(U)$ denote the space of all continuous, compactly supported functions $F$
that assign to each $(x,P)\in G_m(U)$ a vector $F(x,P)$ in $\Tan(U,x)$.
\end{definition}
\end{comment}

In the following theorem, $1_K$ denotes the characteristic function of the set $K$.
Thus if $M$ is a Radon measure on $U$ and if $p<\infty$, then
\[
   \|Y1_K\|_{L^p(M)} = \left( \int_K |Y|^p \,dM \right)^{1/p}.
\]
Similarly,
\[
   \|Y 1_K\|_{L^\infty(M)}
\]
is the essential supremum of $|Y|$ on the set $K$ with respect to the measure $M$.

\begin{theorem}\label{varifold-theorem}
Let $V_i$ and $V$ be rectifiable $m$-varifolds in $U$
such that $V_i\rightharpoonup V$.  Let $M_i$ and $M$ be the associated Radon measures on $U$.
Suppose that $p\in (1,\infty]$ and that $Y_i$ is a Borel vectorfield on $U$ such that
\[
   c_K := \sup_i \|Y_i 1_K\|_{L^p(M_i)} < \infty
\]
for every $K\subset\subset U$.
Then, after passing to a subsequence, there is a Borel vectorfield $Y$ on $U$ in $L^p_\textnormal{loc}(M)$
such that
\[
   \int X(x,\Tan(M_i(x))\cdot Y_i(x) \,dM_i(x) \to \int X(x,\Tan(M,x))\cdot Y(x)\,dM(x)
\]
for every  $X\in \Xx_m(U)$.
\end{theorem}

\begin{proof}
Define 
\begin{align*}
&L_i: \Xx_m(U) \to \RR, \\
&L_i(X) = \int X(x,\Tan(M_i(x))\cdot Y_i(x)\,dM_i(x).
\end{align*}
If $K\subset\subset U$ and if $X\in \Xx_m(U)$ is supported in $\{(x,P):x\in K\}$, then
by H\"older's Inequality,
\begin{equation}\label{M-version}
  |L_i(X)| 
  \le
  c_K \left( \int |X(x,\Tan(M_i,x))|^q\,dM_i(x)\right)^{1/q},
\end{equation}
where $q=p/(p-1)$, or, equivalently, 
\begin{equation}\label{V-version}
|L_i(X)| \le c_K \left( \int |X(x,P)|^q \,dV_i(x,P) \right)^{1/q}.
\end{equation}
From~\eqref{M-version}, we see that
\begin{equation}\label{sup-bound}
  |L_i(X)| \le c_K d_K{}^{1/q} \sup|X|
\end{equation}
where $d_K:=\sup_i M_i(K)$. Note that $d_K<\infty$ by weak convergence of $M_i$ to $M$.

By~\eqref{sup-bound} and Banach-Alaoglu, we can assume,
after passing to a subsequence, that there is an $L: \Xx_m(U)\to \RR$ such that
\[
   L_i(X) \to L(X) \quad\text{for every $X\in \Xx_m(U)$}.
\]
Letting $i\to \infty$ in~\eqref{V-version} gives
\begin{equation}
 |L(X)| \le c_K \left( \int |X(x,P)|^q\,dV(x,P) \right)^{1/q}.
\end{equation}

By the Riesz Representation Theorem, there is an $M$-measurable vectorfield $\tY(x,P)$
such that
\begin{equation}\label{upstairs-version}
   L(X) = \int X(x,P)\cdot \tY(x,P)\,dV(x,P)
\end{equation}
for every $X\in \Xx_m(U)$.  Since $V$ is rectifiable and since $M$ is the associated Radon measure on $U$,
we can rewrite~\eqref{upstairs-version} as
\[
  L(X) = \int X(x,\Tan(M,x))\cdot \tY(x,\Tan(M,x))\,dM(x).
\]
Thus if we set $Y(x) = \tY(x,\Tan(M,x))$, then we have
\[
  L(X) = \int X(x,\Tan(M,x))\cdot Y(x) \,dM(x),
\]
as desired.
\end{proof}

\begin{corollary}\label{varifold-corollary}
In Theorem~\ref{varifold-theorem}, if $Y_i(x)$ is perpendicular to $\Tan(M_i,x)$ for $M_i$-almost every $x$,
then $Y(x)$ is perpendicular to $\Tan(M,x)$ for $M$-almost every $x$.
\end{corollary}

\begin{proof}
Suppose $X\in \Xx_m(U)$.  Let
\[
   \tilde X(x,P) = \Pi_{P} X(x,P).
\]
Then $\tilde X$ is also in $\Xx_m(U)$.  Hence
\[
   \int \tilde X(x,\Tan(M_i,x)) \cdot Y_i(x) \,dM_i(x) \to \int \tilde X(x,\Tan(M,x))\cdot Y(x)\,dM(x),
\]
i.e,
\[
  \int \Pi_{\Tan(M_i,x)}X(x)\cdot Y(x) \,dM_i(x) \to \int \Pi_{\Tan(M,x)}X(x)\cdot Y(x)\,dM(x)
\]
The left hand side is $0$, so
\[
   \int \Pi_{\Tan(M,x)}X(x)\cdot Y(x)\,dM(x) = 0
\]
or, equivalently,
\[
  \int X(x)\cdot \Pi_{\Tan(M,x)}Y(x)\,dM(x)= 0
\]
for all $X\in \Xx(U)$.  Since $\Xx(U)$ is dense in $L^q_\textnormal{loc}(M)$
 (cf.~\cite{ilmanen-elliptic}*{\S7.4}), it follows that
\[
   \Pi_{\Tan(M,x)}Y(x)=0
\]
for $M$-almost every $x$.
\end{proof} 

\begin{theorem}\label{radon-measure-theorem}
Suppose $\beta_i$ $(i=1,2,\dots)$ and $\beta$ are Radon measures on $U$
such that $\beta_i$ converges to $\beta$.  Suppose that $p\in (1,\infty]$ and that 
$Y_i$ is a Borel vectorfield on $U$ 
such that
\[
    \sup_i \|1_KY_i\|_{L^p(\beta_i)} < \infty
\]
for every $K\subset\subset U$.
Then (after passing to a subsequence) there is a Borel vectorfield $Y$ on $U$ in $L^p_\textnormal{loc}(M)$
such that 
\[
    \int X\cdot Y_i\,d\beta_i \to \int X\cdot Y\,d\beta
\]
for all $X$ in $\Xx(U)$.
\end{theorem}

The proof is essentially the same as the proof of Theorem~\ref{varifold-theorem}, 
except that we work in $U$ rather than in $G_m(U)$.

\newcommand{\Vv}{\mathcal{V}}
\newcommand{\Var}{\operatorname{Var}}

\newcommand{\Tr}{\operatorname{Trace}}

\section{A Varifold Closure Theorem}

Let $U$ be an open subset of a smooth Riemannian manifold,
let $\MM(U)$ be the set of all Radon measures on $U$, and
let $\MM_k(U)$ be the set of Radon measures associated to $k$-dimensional rectifiable varifolds in $U$.
Equivalently, $\MM_k(U)$ is the set of Radon measures $M$ such that 
\begin{enumerate}[\upshape (i)]
\item $M(U\setminus S)=0$ for some countable union $S$ of $k$-dimensional $C^1$ submanifolds of $U$, and
\item $M$ is absolutely continuous with respect to $\Hh^k$.
\end{enumerate}
Let $I\MM_k(U)$ be the set of $M\in \MM_k(U)$ such that $\Theta(M,x)$ is an integer for $M$-almost every $x$.

If $M\in \MM_k(U)$, we let $\Var(M)$ be the associated $k$-dimensional varifold in $U$.
Thus $M\in I\MM_k(U)$ if and only if $\Var(M)$ is an integral varifold.

Now suppose that $M\in \MM_k(U)$ and that $\Var(M)$ has bounded first variation.  
Then there exist an $M$-locally integrable vectorfield $H(\cdot)=H(M,\cdot)$, a Radon
measure $\beta(M)$ that is singular with respect to $M$, and a $\beta(M)$-locally integrable unit vectorfield 
$\eta(\cdot)=\eta(M,\cdot)$ with the following property: if $X$ is any compactly supported, $C^1$ vectorfield on $U$,
then
\begin{equation}\label{divergence-formula}
   \int_M\Div_M X\,dM = - \int H\cdot X\,dM + \int X\cdot \eta\,d\beta.
\end{equation}

\begin{comment}
If $[a,b]\subset \RR$, 
we let $\Ff(U\times[a,b])$ be the set of functions $u: U\times[a,b]\to \RR$ such
that $u$ is lipschitz, $\{u>0\}$ has compact closure, and $u$ is $C^1$ on $\{u>0\}$.
\end{comment}

\begin{definition}\label{Vm-definition}
If $\Gamma$ is a properly embedded $(m-1)$-dimensional submanifold of $U$,
then $\Vv_m(U,\Gamma)$ is the space of $M\in I\MM_m(U)$ 
such that $M$ has bounded first variation and such
such that
\begin{enumerate}
\item $\beta(M) \le \Hh^{m-1}\llcorner \Gamma$.
\item\label{perpendicular-item} $H(\cdot)$ and $\Tan(M,\cdot)$ are perpendicular $M$-almost everywhere.
\end{enumerate}
\end{definition}

(As mentioned in the introduction,  the orthogonality condition~\eqref{perpendicular-item} in Definition~\ref{Vm-definition}
is superfluous according to a theorem of Brakke~\cite{brakke}*{\S5}, but the proof of that
theorem is rather difficult.  Including Condition~\eqref{perpendicular-item}  in the definition makes
that theorem unnecessary for us.)

Let
\begin{equation}\label{expression-for-nu}
  \nu(M,x) = \lim_{r\to 0} \frac1{\omega_{m-1}r^{m-1}}\int_{\BB(x,r)}\eta(\cdot)\,d\beta
\end{equation}
where the limit exists, and let $\nu(M,x)=0$ where the limit does not exist.
 Note that the limit exists $\Hh^{m-1}$ almost everywhere. Note also that
 we can rewrite~\eqref{divergence-formula} as
\begin{equation}\label{divergence-rewritten}
    \int \Div_MX\,dM = -\int H\cdot X\,dM + \int_\Gamma \nu\cdot X\,d\Hh^{m-1}
\end{equation}
or (using the notational conventions described in Section~\ref{notation-section}) as
\begin{equation*}\label{divergence-rewritten}
    \int \Div_MX\,dM = -\int H\cdot X\,dM + \int \nu\cdot X\,d\Gamma.
\end{equation*}
\begin{remark}\label{nu-equivalence-remark}
{\rm
The condition that  $\beta\le \Hh^m\llcorner \Gamma$ is equivalent to the condition  that $|\nu(x)|\le 1$
for $\Hh^{m-1}$ almost every $x\in \Gamma$.
}
\end{remark}

\begin{remark}\label{nu-perpendicular}
{\rm
If $M\in \Vv_m(U,\Gamma)$, then $\nu(M,\cdot)$ is perpendicular to $\Gamma$
at almost every point of $\Gamma$ by~\cite{allard-boundary}*{\S3.1}.
}
\end{remark}

In the following theorem, we write $H_i(\cdot)$ and $H(\cdot)$ for $H(M_i,\cdot)$ and $H(M,\cdot)$,
and $\nu_i(\cdot)$ and $\nu(\cdot)$ for $\nu(M_i,\cdot)$ and $\nu(M,\cdot)$.

\begin{theorem}[Varifold Closure Theorem]\label{varifold-closure-theorem}
Suppose for $i=1,2,\dots$ that $M_i\in \Vv_m(U,\Gamma_i)$, where the $\Gamma_i$
are smooth $(m-1)$-dimensional submanifolds of $U$ that converge in $C^1$
to a smooth manifold $\Gamma$.
Suppose that the $M_i$ converge to a Radon measure $M$ and that
\begin{equation*}\label{varifold-closure-d}
  d_K:= \sup_i \int_K |H_i|^2\,dM_i < \infty
\end{equation*}
for every $K\subset\subset U$.
Then
\begin{enumerate}[\upshape(1)]
\item\label{uniform-bounded-variation}
For every $K\subset\subset U$, 
\[
\sup_i \left(\int_K|H_i|\,dM_i + \beta(M_i)(K)\right) < \infty.
\]
\item\label{M-good-theorem} $M\in I\MM_m(U)$ and $\Var(M_i)$ converges to $\Var(M)$.
Thus if $f:G_m(U)\to \RR$ is continuous and compactly supported, then
\begin{equation*}
    \int f(x, \Tan(M_i,x))\,dM_i(x) \to \int f(x , \Tan(M,x))\,dM(x)
\end{equation*}
\item\label{H-good-theorem} If $X\in \Xx_m(U)$,
 then
   \begin{equation*}
     \int X(x,\Tan(M_i,x))\cdot H_i(x)\,dM_i \to \int X(x,\Tan(M,x))\cdot H(x)\,dM(x).
   \end{equation*}
\item\label{nu-good-theorem} If $Z\in \Xx_{m-1}(U)$, then
  \begin{equation*}
      \int Z(x,\Tan(\Gamma_i,x))\cdot \nu_i(x)\,d\Gamma_i(x) 
      \to 
      \int Z(x,\Tan(\Gamma_i,x))\cdot\nu(x) \,d\Gamma(x).
  \end{equation*}
\item\label{belongs} $M\in \Vv_m(U,\Gamma)$.
 \end{enumerate}
\end{theorem}

\begin{proof}
Since the $M_i$ converge to $M$, 
\begin{equation}\label{varifold-closure-c}
  c_K:=\sup_i M_i(K) < \infty
\end{equation}
for every $K\subset\subset U$.
Thus
\begin{align*}
\int_K|H_i|\,dM_i 
&\le
(M_i(K))^{1/2} \left(\int_K|H_i|^2\,dM_i \right)^{1/2} 
\\
&\le
(c_Kd_K)^{1/2}
\end{align*}
Also,
\[
 \sup_i \beta(M_i)(K) \le \sup_i\Hh^{m-1}(\Gamma_i\cap K) < \infty
\]
since $\Gamma_i$ converges in $C^1$ to $\Gamma$.  
This proves Assertion~\eqref{uniform-bounded-variation}.

By Assertion~\eqref{uniform-bounded-variation} 
and by Allard's Closure Theorem for Integral
 Varifolds~(\cite{allard-first-variation}*{Theorem 6.4} 
 or~\cite{simon-gmt}*{\S42.8}
 or~\cite{simon-new-gmt}*{chapter 8, \S5.9}),
the varifolds $\Var(M_i)$ converge (after passing to a subsequence) to an integral varifold $V$
of bounded first variation.
Note that $\mu_V=M$, so the limit $V=\Var(M)$ does not depend
on the choice of subsequence.  Thus the original sequence $\Var(M_i)$ converges to $\Var(M)$.
 Thus we have proved Assertion~\eqref{M-good-theorem} of the theorem.

By Theorem~\ref{varifold-theorem}, every sequence of $i$ tending to infinity has
a  subsequence $i(j)$ for which there exist
an $M$-measurable
vectorfield $\tilde H$ and a $\Gamma$-measurable vectorfield $\tilde\nu$ such that
\begin{equation}\label{H-good}
\begin{aligned}
  \int X(x,\Tan(M_{i(j)},x))\cdot H_{i(j)}&(x) \,dM_{i(j)}(x) 
  \\
  & \to \int X(x,\Tan(M,x))\cdot \tilde H(x)\,dM(x)
\end{aligned}
\end{equation}
for every $X\in \Xx_m(U)$ and 
\begin{equation}\label{nu-good}
\begin{aligned}
  \int Z(x,\Tan(\Gamma_{i(j)},x))\cdot \nu_{i(j)}(&x)\,d\Gamma_{i(j)}(x)
  \\
  &\to
  \int Z(x,\Tan(\Gamma,x))\cdot \tilde\nu(x)\,d\Gamma(x)
\end{aligned}
\end{equation}
for every $Z\in \Xx_{m-1}(U)$.
Furthermore, by Corollary~\ref{varifold-corollary}, 
the perpendicularity almost everywhere of $H(M_i,\cdot)$ and $\Tan(M_i,\cdot)$ implies
the perpendicularity almost everywhere of $\tilde H(\cdot)$ and $\Tan(M,\cdot)$. 
  Also, from~\eqref{nu-good} (and Remark~\ref{nu-equivalence-remark}) we see that $|\tilde\nu(\cdot)|\le 1$
almost everywhere with respect to $\Gamma$.

For every $C^1$, compactly supported vectorfield $X$ on $U$, we have
\[
   \int \Div_{M_{i(j)}}X \,dM_{i(j)} 
   = 
   -\int H_{i(j)}\cdot X \,dM_{i(j)}(\cdot) 
   + \int \nu_{i(j)} \cdot X \,d\Gamma_{i(j)}(\cdot).
\]
By the convergence $\Var(M_{i(j)})$ to $\Var(M)$ and by~\eqref{H-good} and~\eqref{nu-good}, 
it follows that
\[
   \int \Div_{M}X \,dM
   = 
   -\int \tilde H \cdot X \,dM_i
   + \int \tilde\nu \cdot X\,d\Gamma.
\] 
Consequently, $M\in \Vv_m(U,\Gamma)$, $\tilde H(\cdot)=H(M,\cdot)$ and $\tilde\nu(\cdot)=\nu(M,\cdot)$.
We passed to a subsequence $i(j)$, but since the limits $H(M,\cdot)$ and $\nu(M,\cdot)$ are independent of the
choice of subsequence, in fact~\eqref{H-good} and~\eqref{nu-good} hold for the original sequence.
\end{proof}

\section{Brakke Flows with Boundary}\label{brakke-flow-section}

\begin{definition}\label{brakke-flow-definition}
An $m$-dimensional {\bf integral Brakke flow with boundary} in $U$ is a pair $(M(\cdot),\Gamma)$
where $\Gamma$ is a smooth, properly embedded $(m-1)$-dimensional submanifold of $U$
and where
\[
   t\in I\mapsto M(t)
\]
is a Borel map from an interval $I$ to the space $\MM(U)$ of Radon measures in $U$ such that
\begin{enumerate}[\upshape (1)]
\item For almost every $t\in I$, $M(t)$ is in $\Vv_m(U,\Gamma)$ (see Definition~\ref{Vm-definition}).
\item\label{finiteness-item}
If $[a,b]\subset I$ and if $K\subset U$ is compact, then
\[
   \int_a^b\int_K (1+|H|^2)\,dM(t)\,dt < \infty.
\]   
\item\label{defining-inequality}
If $[a,b]\subset I$ and if $u$ is a nonnegative, compactly supported, $C^2$ function on $U\times[a,b]$,
 then
\begin{equation*}
 (Mu)(a) - (Mu)(b) 
  \ge
\int_a^b \int \left( u|H|^2 - H\cdot \nabla u - \pdf{u}{t} \right)\,dM(t)\,dt.
\end{equation*}
\end{enumerate}
We also say that ``$M(\cdot)$ is an integral Brakke flow with boundary $\Gamma$''.
\end{definition}

By~\eqref{finiteness-item}, the integral in~\eqref{defining-inequality} is finite.

(The condition that $t\mapsto M(\cdot)$ is a Borel map is equivalent
to the condition that $t\mapsto M(t)f$ is a Borel map for every continuous, compactly supported function $f$ on $U$.)

\begin{proposition}\label{more-general-u-proposition}
If $t\in I\mapsto M(t)$ is a Brakke flow with boundary $\Gamma$, then 
the defining inequality~\eqref{defining-inequality} in Definition~\ref{brakke-flow-definition}
 holds for every nonnegative, compactly supported,
Lipschitz function $u$ on $U$ that is $C^1$ on $\{u>0\}$.
\end{proposition}

\begin{proof}
Approximate $u$ by $C^2$ functions $u_n$ and use the Dominated Convergence Theorem.
\end{proof}

\begin{lemma}\label{trace-lemma}
Suppose that $M$ is a rectifiable varifold of bounded first variation and that $u$ is a nonnegative, compactly supported,
Lipschitz function such that $u|\{u>0\}$ is $C^2$ and such that
\[
    \sup_{u(x)>0} |\nabla^2u(x)|<\infty.
\]
Then
\[
  \int \Div_M\nabla u\,dM \le  - \int H(M,\cdot)\cdot \nabla u \,dM + \int \eta(M,\cdot) \cdot \nabla u \,d\beta(M).
\]
\end{lemma}

\begin{proof}
Let $\phi:\RR\to\RR$ be a smooth increasing function such that $\phi(x)=0$ for $x<1$, $\phi(x)=x-1$ for $x\ge 3$,
and such that $\phi''\ge 0$ everywhere.  Let $\kappa>0$,
 and apply the Divergence Theorem to $\kappa^{-1}\phi(\kappa u)$:
\begin{align*}
  \int \Div_M\nabla (\kappa^{-1}\phi(\kappa u))\,dM 
  &= -\int H\cdot \nabla (\kappa^{-1}\phi(\kappa u))\,dM + \int \eta\cdot \nabla(\kappa^{-1}\phi(\kappa u))\,d\beta
\end{align*}
Now $\nabla(\kappa^{-1}\phi(\kappa u))= \phi'(\kappa u)\nabla u$, so  
\begin{align*}
\Div_M\nabla(\kappa^{-1}\phi(\kappa u))
&=
\Div_M( \phi'(\kappa u) \nabla u)
\\
&=
\phi''(\kappa u) |\nabla_Mu|^2 \kappa + \phi'(\kappa u)\Div_M\nabla u
\\
&\ge  \phi'(\kappa u)\Div_M\nabla u.
\end{align*}
Thus
\begin{align*}
  \int \phi'(\kappa u)\Div_M\nabla u\,dM 
  &\le  -\int \phi'(\kappa u) H\cdot \nabla u\,dM + \int \phi'(\kappa u) \eta\cdot \nabla u\,d\beta
\end{align*}
Now let $\kappa\to\infty$ and use the Dominated Convergence Theorem.
\end{proof}

If $S$ is an $n\times n$ symmetric matrix with eigenvalues $\lambda_1\le \lambda_2 \le\dots\le\lambda_n$,
let $\Tr_k(S)=\sum_{i=1}^k\lambda_i$.  Thus if $M$ is an $m$-dimensional submanifold (or, more generally,
if $M$ is in $\MM_m(U)$), then
\[
   \Div_M\nabla u \ge \Tr_m(\nabla^2u).
\]

\begin{corollary}\label{trace-corollary}
If $t\in [a,b] \mapsto M(t)$ is a Brakke Flow with boundary $\Gamma$, 
if 
\[
   f:U\times[a,b]\to \RR
\]
is a nonnegative, $C^2$ function with $\overline{\{f>0\}}$ compact, and if $u:= 1_{f\ge 0}f$, then
\begin{align*}
&(Mu)(a)-(Mu)(b)
\\
&\qquad\ge
\int_a^b\int \left(u|H|^2 + \Div_M\nabla u  - \pdf{u}t \right)\,dM(t)\,dt - \int_a^b\int \nu\cdot\nabla u\,d\Gamma\,dt
\\
&\qquad\ge
\int_a^b\int \left(u|H|^2 + \Tr_m(\nabla^2 u) - \pdf{u}t \right)\, dM(t)\, dt - \int_a^b\int\nu\cdot\nabla u\,d\Gamma\,dt.
\end{align*}
\end{corollary}

\begin{proof}
This follows immediately from Proposition~\ref{more-general-u-proposition} and
Lemma~\ref{trace-lemma}. 
\end{proof}

As a special case of Corollary~\ref{trace-corollary}, we have

\begin{theorem}\label{trace-theorem}
Let $t\in [0,T]\mapsto M(t)$ be a Brakke Flow with boundary $\Gamma$.
Suppose that
\begin{align*}
&\text{$\overline{\BB(x,R)}$ is a compact subset of $U$}, \\
&\text{$\dist(\cdot,x)^2$ is smooth on $\overline{\BB(x,R)}$, and} \\
&\text{$\Tr_m(-\nabla^2(\dist(\cdot,x)^2))\ge -4m$ on $\BB(x,R)$.}
\end{align*}
Let $u= (R^2-\dist(\cdot,x)^2- 4mt)^+$.
Then for $t\in [0,T]$,
\[
   (Mu)(t) \le  (Mu)(0) + 2tR \Hh^{m-1}(\Gamma\cap \BB(x,R)).
\]
\end{theorem}

Note that $\Tr_m(-\nabla^2(\dist(\cdot,x)^2) = -2m$ at $x$, so the
hypotheses in Theorem~\ref{trace-theorem} are satisfied if $R$ is sufficiently small. 
\begin{comment}
(For example,
by standard comparison theorems, the trace hypothesis holds if the sectional
curvatures in $\BB(x,R)$ are $\ge - 1/R^2$.)
\end{comment}

\begin{theorem}\label{H-squared-control}
Let $t\in I\mapsto M(t)$ be a Brakke Flow with boundary $\Gamma$.
Let $u$ be a nonnegative, compactly supported, $C^2$ function on $U$.
Then for $[a,b]\subset I$, 
\begin{align*}
  \frac12 \int_a^b\int u |H|^2\,dM(t)\,dt 
  &\le  M(a)u- M(b)u  \\
  &\quad + (b-a)(\max|\nabla^2u|)K_{[a,b]},
\end{align*}
where 
\[
   K_{[a,b]} := \sup_{t\in [a,b]}M(t)(\spt u).
\]
\end{theorem}

\begin{proof}
By~\cite{ilmanen-elliptic}*{Lemma~6.6},
\[
     |\nabla u|^2 \le  2\,|u| \max{|\nabla^2u|}.
\]
Thus wherever $u>0$, 
\begin{align*}
u|H|^2 - H\cdot \nabla u 
&=
\frac12 u |H|^2 + \frac12|u^{1/2}H - u^{-1/2}(\nabla u)|^2 - \frac12\frac{|\nabla u|^2}{|u|}
\\
&\ge
\frac12 u |H|^2  - \max|\nabla^2u|
\end{align*}
Consequently, 
\begin{align*}
M(a)u - M(b)u
&\ge
\int_a^b \int (u|H|^2 - H\cdot \nabla u)\,dM_i(t)\,dt
\\
&\ge
\frac12 \int_a^b\int u|H|^2\,dM_i(t)\,dt -  \max|\nabla^2u| \int_a^b M(\spt u)\,dt 
\\
&\ge
\frac12 \int_a^b\int u|H|^2\,dM(t)\,dt -  (b-a) \max|\nabla^2u| K_{[a,b]}.
\end{align*}
\end{proof}

\begin{corollary}\label{measure-increasing-corollary}
If $K_I<\infty$, then
\[
  t\in I \mapsto M(t)u -  (\max|\nabla^2u|) K_I t
\]
is a non-increasing function of $t$.
\end{corollary}

\section{Monotonicity with Boundary in a Manifold}\label{monotonicity-section}

Now consider mean curvature flow in a smooth Riemannian manifold $N$.
We embed $N$ isometrically in a Euclidean space $\RR^d$.
By spacetime translation, it suffices to consider monotonicity about the origin
in spacetime.  By parabolic scaling, we can assume that $N$ is properly embedded
in an open subset of $\RR^d$ that contains $\BB^d(0,1)$.
If $M$ is an $m$-dimensional submanifold of $N$, we let $H$ be the mean curvature as a submanifold
of $\RR^d$, and we let $H_N$ and $H_{N^\perp}$ be the projections of 
$H$ to $\Tan(N,\cdot)$ and to $\Tan(N^\perp,\cdot)$.
Thus $H_N$ is the mean curvature of $M$ as a submanifold of $N$.

For $x\in\RR^d$ and $t<0$, let
\begin{equation*}
  \rho(x,t) = \frac1{(4\pi |t|)^{m/2}} \exp\left( - \frac{|x|^2}{4|t|} \right)
\end{equation*}
and
\begin{equation*}
  \hat\rho(x,t) = \phi(|x|) \rho(x,t),
\end{equation*}
where $\phi$ is a a smooth function compactly supported in $[0,1)$ such that $\phi=1$ on $[0,1/2]$, and $\phi'\le 0$.

A straightforward calculation (see~\cite{kasai-tonegawa}*{\S6.1}) shows that 
\begin{equation}\label{K-constant-monotonicity}
K:=\sup\left| \pdf{\hat\rho}t + \Div_M\nabla \hat\rho +  \frac{|\nabla^\perp \hat\rho|^2}{\hat\rho}\right| <\infty
\end{equation}
and that if $\Gamma$ is a smooth, properly embedded, $(m-1)$-dimensional submanifold of $\overline{\BB(0,1)}$, then
\begin{equation}\label{C-gamma-constant-monotonicity}
   C_\Gamma:= \sup_{t<0} \int  |(\nabla\hat\rho(\cdot,t))_{\Gamma^\perp}|\, d\Gamma  < \infty.
\end{equation}

%\marginpar{Maybe make an integral statement for $C_\Gamma$?}

\begin{theorem}[Huisken Monotonicity]\label{monotonicity-theorem}
Let $U$ be an open subset of $\RR^d$ that contains $\BB^d(0,1)$ and 
 $N$ be a smooth, properly embedded submanifold of $U$.
Let $\Gamma$ be a smooth, properly embedded $(m-1)$-dimensional submanifold of $U$.
Let $t\in I\mapsto M(t)$ be a Brakke flow in $N$ with boundary $\Gamma$.
Suppose that
\begin{equation}\label{C-constant-monotonicity}
    M(t)\BB(0,1)\le C
\end{equation}
for $t\in I$ and that the norm of the second fundamental form of $N$ is bounded by $A$.
Then for $a,b\in I$ with $a\le b< 0$, 
\begin{equation}\label{monotonicity-inequality}
\begin{aligned}
(M\hat\rho)(a) - (M\hat\rho)(b)
&\ge
\int_a^b \int \left| H_N - \frac{(\nabla^\perp\hat\rho)_N}{\hat\rho}\right|^2 \hat\rho\,dM(t)\,dt
\\
&\qquad
+ \int_a^b \int \nu_M\cdot \nabla\hat\rho \,d\Gamma\,dt
\\
&\qquad
-  mA^2 \int_a^b \int \hat\rho\,dM(t)\,dt
\\
&\qquad
-
CK(b-a)
\end{aligned}
\end{equation}
where $C$ and $K$ are as in~\eqref{K-constant-monotonicity} 
and~\eqref{C-constant-monotonicity}.
Furthermore,
\[
    e^{-mA^2t} \left( (M\hat\rho)(t) + \int_{\tau=T_0}^t \int \nu_M\cdot \nabla\hat\rho \, d\Gamma\,d\tau - CKt \right)
\]
is a decreasing function of $t$ for $t<0$ in $I$.
\end{theorem}

\begin{proof}
\begin{equation}\label{M-portion}
\begin{aligned}
&(M\hat\rho)(a)  - (M\hat\rho)(b)
\\
%zeroth
&\quad\ge
\int_a^b 
\int \left(\hat\rho |H_N|^2 -  H_N\cdot\nabla\hat\rho - \frac{\partial\hat\rho}{\partial t} \right) \, dM(t)\,dt
\\
%first
&\quad \ge
\int_a^b 
\int \left(\hat\rho |H_N|^2 -  H_N\cdot\nabla\hat\rho 
+ \Div_M\nabla\hat\rho + \frac{|\nabla^\perp\hat\rho|^2}{\hat\rho} - K1_\BB\right) \, dM(t)\,dt
\\
%second
&\quad=
\int_a^b \int \left(\hat\rho |H_N|^2 -  H_N\cdot\nabla\hat\rho - H\cdot\nabla\hat\rho 
    + \frac{|\nabla^\perp\hat\rho|^2}{\hat\rho} - K\,1_\BB \right) \, dM(t)\,dt
\\
&\qquad\qquad+
\int_a^b\int \nu_M\cdot\nabla \hat\rho \,d\Gamma\,dt
\\
%third
&\quad\ge
\int_a^b \int_\BB \left(\hat\rho |H_N|^2 -  H_N\cdot\nabla\hat\rho - H\cdot\nabla\hat\rho 
    + \frac{|\nabla^\perp\hat\rho|^2}{\hat\rho} \right) \, dM(t)\,dt
\\
&\qquad\qquad+
\int_a^b\int \nu_M\cdot\nabla \hat\rho \,d\Gamma\,dt  - KC(b-a).
\end{aligned}
\end{equation}
We rewrite the penultimate the integrand in~\eqref{M-portion} as follows, using the orthogonality
of the mean curvature:
\begin{align*}
&\hat\rho |H_N|^2 -  H_N\cdot\nabla\hat\rho - H\cdot\nabla\hat\rho 
    + \frac{|\nabla^\perp\hat\rho|^2}{\hat\rho}
\\
&\quad=
\hat\rho |H_N|^2 -  2H_N\cdot\nabla\hat\rho - H_{N^\perp}\cdot\nabla\hat\rho 
    + \frac{|\nabla^\perp\hat\rho|^2}{\hat\rho} 
\\
&\quad=
\hat\rho |H_N|^2 -  2H_N\cdot\nabla^\perp\hat\rho - H_{N^\perp}\cdot\nabla\hat\rho 
    + \frac{|\nabla^\perp\hat\rho|^2}{\hat\rho} 
\\
&\quad=
\hat\rho |H_N|^2 -  2H_N\cdot(\nabla^\perp\hat\rho)_N  
               + \frac{|(\nabla^\perp\hat\rho)_N|^2}{\hat\rho} 
  - H_{N^\perp}\cdot(\nabla\hat\rho)_{N^\perp}
    + \frac{|(\nabla\hat\rho)_{N^\perp}|^2}{\hat\rho} 
\\
&\quad=
 \left| H_N - \frac{(\nabla^\perp\hat\rho)_N}{\hat\rho} \right|^2\hat\rho   
                + \left| \frac12 H_{N^\perp} - \frac{(\nabla\hat\rho)_{N^\perp}}{\hat\rho} \right|^2 \hat\rho               
 - \frac14 |H_{N^\perp}|^2 \hat\rho  
\\
&\quad\ge
 \left| H_N - \frac{(\nabla^\perp\hat\rho)_N}{\hat\rho} \right|^2\hat\rho   - mA^2\hat\rho
\end{align*}
since $|H_{N^\perp}|^2\le mA^2$.
Substituting this into~\eqref{M-portion} gives~\eqref{monotonicity-inequality}.

Now let
\begin{equation}\label{f-monotonicity}
    f(t) = (M\hat\rho)(t) + \int_{\tau=T_0}^t \nu_M\cdot \nabla\hat\rho  \, d\Gamma\,d\tau - CKt.
\end{equation}
By~\eqref{monotonicity-inequality}, we have
\[
  f'(t) \le  m A^2 f(t)
\]
 in the distributional sense, which immediately implies that 
\[
  \frac{d}{dt} \left( e^{- m A^2t} f(t)\right) \le 0.
\]
\end{proof}

\begin{corollary}
The quantity $(M\hat\rho)(t)$ has a finite limit as $t\to 0$.
\end{corollary}

\begin{proof}
Since $e^{-mA^2t}f(t)$ is in a non-increasing function of $t$ (where $f$ is given by~\eqref{f-monotonicity}),
$\lim_{t\uparrow}f(t)$ exists and is in $[-\infty,\infty)$.  
By~\eqref{C-gamma-constant-monotonicity},
\[
   \lim_{t\uparrow 0} \int_{\tau=T_0}^t \int \nu_M\cdot \nabla\hat\rho\,d\Gamma\,d\tau
\]
exists and is finite. Thus $\lim_{t\uparrow 0}(M\hat\rho)(t)$ exists and is $<\infty$.
Since $(M\hat\rho)(t)\ge 0$, the limit is $\ge 0$, and thus is a finite, nonnegative number.
\end{proof}

\begin{definition}\label{gauss-density-definition}
The {\bf Gauss density} of $M(\cdot)$ at $(0,0)$
is 
\[
  \Theta(M(\cdot),(0,0)) 
  =
  \lim_{t\uparrow 0} (M\hat\rho)(t)
\]
The {\bf extended Gauss density} of $M(\cdot)$ at $(0,0)$ is
\[
\Theta_e(M(\cdot),(0,0))
=
\begin{cases}
\Theta(M(\cdot), (0,0)) &\text{if $0$ is not in $\Gamma$, and}
\\
\Theta(M(\cdot),(0,0)) + \frac12  &\text{if $0\in \Gamma$}.
\end{cases}
\]
\end{definition}
 It is straightforward to prove that the Gauss
density does not depend on the isometric embedding of $N$ into $\RR^d$ or
on the choice of the cutoff function $\phi$.

\section{Monotonicity with Boundary in Euclidean Space}\label{euclidean-monotonicity-section}

The monotonicity inequality becomes simpler for mean curvature flow with fixed boundary in a Euclidean space.
(The material in this section and in Sections~\ref{entropy} and~\ref{maximal-density-section}  is not used in the rest of the paper.)

Let $\Gamma$ be a smooth, properly embedded $(m-1)$-dimensional manifold in $\RR^n$.
For $v\in \RR^n$, the {\bf exterior cone} over $\Gamma$ with vertex $v$ is
\[
   \{v + s(x-v): x\in \Gamma, \, s\ge 1\}.
\]
The {\bf multiplicity} $\theta(p)=\theta_{\Gamma,v}(p)$ of the exterior cone at a point $p\in \RR^n$
is the number of points $(x,s)\in \Gamma\times [1,\infty)$ such that
\[
      v + s(x-v)=p.
\]
The exterior cone (counting multiplicity) determines a Radon measure $E=E_{\Gamma,v}$ on $\RR^n$:
\[
 dE_{\Gamma,v} = \theta_{\Gamma, v} \,d\Hh^m.
\]

\begin{theorem}\label{extended-monotonicity}
Suppose $t\in [a,b] \mapsto M(t)$ is a $m$-dimensional Brakke flow in $\RR^n$ with boundary $\Gamma$.  
Let $v\in \RR^n$, $t_0\ge b$, and
\[
\psi(x,t) = \psi_{v,t_0}(x,t) = \frac1{(4\pi (t_0-t))^{m/2}}  \exp\left( \frac{- |x - v|^2}{4 (t_0-t)} \right).
\]
Then
\[
    (M(t) + E_{\Gamma,v}) \psi
\]
is a decreasing function of $t$ for $t\in [a,b]$.  Indeed,
\begin{equation}\label{monotonicity-with-exterior-cone}
\begin{aligned}
&(M(a)+ E_{\Gamma,v})\psi(\cdot,a) - (M(b) + E_{\Gamma,v})\psi(\cdot,b)
\\
&\qquad =
\int_a^b  \int  \left| H - \frac{\nabla^\perp\psi}{\psi} \right|^2 \psi \,dM(t) \,dt 
+
\int_a^b \int | (\nu_M + \nu_E)\cdot\nabla \psi |\, d\Gamma\, dt.
\end{aligned}
\end{equation}
Furthermore, if
\[
  (M(a)+ E_{\Gamma,v})\psi(\cdot,a) = (M(b) + E_{\Gamma,v})\psi(\cdot,b),
\]
then for almost every $t\in [a,b]$,
\begin{equation}\label{H-OK}
H= \nabla^\perp\psi /\psi
\end{equation}
 holds $M(t)$-almost everyhwere, and
\begin{equation}\label{nu-OK}
  \nu_{M(t)}(x) = \frac{(x-v)^\perp}{|(x-v)^\perp|}
\end{equation}
holds for almost every $x\in \Gamma$ for which $(x-v)^\perp$ is nonzero,
 where $(x-v)^\perp$ is the projection of $(x-v)$ to $\Tan(\Gamma,x)^\perp$.
\end{theorem}

The theorem says that, although the integral of $\psi(\cdot,t)$ over $M(t)$ need not be decreasing (as a function of  $t$), if we extend $M(t)$ by attaching the exterior cone,
i.e., if we replace $M(t)$ by $M(t)+E_{\Gamma,t}$, then the integral of $\psi(\cdot,t)$ over the extended surface is decreasing.

(This is very analogous to the extended monotonicity formula for minimal surfaces in~\cite{EWW}.)

\begin{proof}
By translating and parabolically dilating, it suffices to consider the case when $(v,t_0)=(0,0)$ and $a< b< 0$.
Thus $\psi=\rho$, where $\rho$ is as in Section~\ref{monotonicity-section}.   For simplicity, we give the proof in the case that $\Gamma$
and the supports of the $M(t)$ lie in a compact set.  (Otherwise, one uses cutoff functions and then lets the cutoff functions tend to the constant function $1$.)

In this case, the inequality~\eqref{monotonicity-inequality} in the Monotonicity Theorem~\ref{monotonicity-theorem} 
becomes equality with $\rho$ in place of $\hat\rho$
and with $A=K=0$:
\begin{equation}\label{interior-thing}
\begin{aligned}
(M\rho)(a) - (M\rho)(b)
&=
\int_a^b \int \left| H  - \frac{(\nabla^\perp\rho)}{\rho}\right|^2 \rho\,dM(t)\,dt
\\
&\qquad
+ \int_a^b \int \nu_M\cdot \nabla\rho \,d\Gamma\,dt.
\end{aligned}
\end{equation}

For notational simplicity, let us assume the exterior cone over $\Gamma$ (with vertex $0$) is embedded, i.e., that
the multiplicity $\theta=\theta_{\Gamma,0}$ is $1$ at all points on the cone.
We will use $E=E_{\Gamma,0}$ to denote both the exterior cone and the associated Radon measure.

A straightforward calculation shows that on $E$,
\[
  \pdf{\rho}t = - \Div_E \nabla\rho.
\]

Thus
\begin{align*}
\frac{d}{dt}E \rho(\cdot,t)
&=
\frac{d}{dt} \int_E \rho \,d\Hh^m
\\
&=
\int_E \pdf{\rho}t  \, d\Hh^m 
\\
&=
-\int_E \Div_E \nabla\rho\,d\Hh^m
\\
&=
 \int_E H_E \cdot \nabla \rho \,d\Hh^m  -  \int_\Gamma \nu_E\cdot\nabla \rho\,d\Hh^{m-1}
\\
&=
- \int_\Gamma \nu_E\cdot \nabla \rho\,d\Hh^{m-1}.
\end{align*}
(Note that $H_E \cdot \nabla \rho\equiv 0$ since $H_E$ is perpendicular to $E$ and $\nabla \rho$ is tangent to $E$
because $E$ is a portion of a cone with vertex at the origin.)
Thus
\begin{equation}\label{exterior-thing}
\begin{aligned}
E\rho(\cdot,a) - E\rho(\cdot,b)
&=   - \int_a^b \frac{d}{dt}E\rho(\cdot,t)\,dt
\\
&=
\int_a^b \int \nu_E\cdot \nabla \rho\,d\Gamma.
\end{aligned}
\end{equation}
Adding~\eqref{interior-thing} and~\eqref{exterior-thing} gives
\begin{equation}\label{monotonicity-with-exterior-cone-2}
\begin{aligned}
&(M(a)+ E_{\Gamma,v})\rho(\cdot,a) - (M(b) + E_{\Gamma,v})\rho(\cdot,b)
\\
&\qquad =
\int_a^b  \int  \left| H - \frac{\nabla^\perp\rho}{\rho} \right|^2 \rho \,dM(t) \,dt 
+
\int_a^b \int  (\nu_M + \nu_E)\cdot\nabla \rho \, d\Gamma\, dt.
\end{aligned}
\end{equation}

Note that $\nabla\rho$ points toward the origin.

\begin{claim}\label{claim}
If $x^\perp=0$ (i.e., if $x\in \Tan(\Gamma,x)$), then
\[
    (u + \nu_E)\cdot \nabla \rho = 0
\]
for all $u\in \Tan(\Gamma,x)^\perp$.  If $x^\perp\ne 0$, then 
\[
  \nu_E = -\frac{x^\perp}{|x^\perp|}
\]
and if $u$ is a vector in $\Tan(\Gamma,x)^\perp$ with $|u|\le 1$, then
\[
   (u + \nu_E)\cdot \nabla \rho \ge 0,
\]
with equality if and only if $u= -\nu_E$.
\end{claim}

The proof of the claim is straightforward vector geometry.

Claim~\ref{claim} implies that in~\eqref{monotonicity-with-exterior-cone-2}, 
we can replace $(\nu_M+\nu_E)\cdot \nabla \rho$ by its absolute value, which gives~\eqref{monotonicity-with-exterior-cone}.
The ``Furthermore'' assertion follows immediately from~\eqref{monotonicity-with-exterior-cone} and Claim~\ref{claim}.
\end{proof}

For the next corollary, we consider the full cone $C_{\Gamma,v}$ over $\Gamma$ with vertex $v$, namely
\[
   \{ v + s(x-v): x\in \Gamma, \, s\ge 0\}.
\]
As above, we make $C_{\Gamma,v}$ into a measure (counting multiplicity) by setting:
\[
    dC_{\Gamma,v} = \theta\,d\Hh^m,
\]
where $\theta(p)$ is the number of points $(x,s)\in \Gamma\times [0,\infty)$ such that
\[
     v + s(x-v) = p.
\]
Since $C_{\Gamma,v}$ is a cone with vertex $v$, the
 ratio $\theta_C:=C_{\Gamma,v}\BB(v,r) / (\omega_mr^m)$ is independent of $r$.

\begin{corollary}\label{new-monotone-corollary}
\begin{align}
(M(b) + E_{\Gamma,v})\psi(\cdot,b) 
&\le \frac{M(a)(\RR^n)}{(4\pi (t_0-a)))^{m/2}} + E_{\Gamma,v}\psi(\cdot,a)
\label{concentrated}
\\
&\le  \frac{M(a)(\RR^n)}{(4\pi (t_0-a)))^{m/2}} + \theta_C.
\label{first-coney}
\end{align}
Furthermore, if the flow is ancient and if $\sup_t M(t)\RR^n< \infty$, then
\begin{equation}\label{second-coney}
(M(b)+E_{\Gamma,v})\psi(\cdot,b) \le \theta_C.
\end{equation}

\end{corollary}

\begin{proof} 
The inequality~\eqref{concentrated} follows from~\eqref{monotonicity-with-exterior-cone}
   since $\psi(\cdot,a)$ attains its maximum value $(4\pi (t_0-a))^{-m/2}$
at the point $v$.
The inequality~\eqref{first-coney} follows since $E_{\Gamma,v}\le C_{\Gamma,v}$ 
 and
since a straightforward calculation shows that
\begin{equation}\label{same}
   C_{\Gamma,v}\psi(\cdot,a) = \theta_C.
\end{equation}
(Alternatively,~\eqref{same} follows immediately from Remark~\ref{weighted-remark} below.)
The  inequality~\eqref{second-coney} follows by letting $a\to -\infty$.
\end{proof}

\section{Entropy}\label{entropy}

In this section, we adapt the concept of entropy in mean curvature flow to mean curvature flow with boundary.

Suppose that $M$ is a properly embedded $m$-dimensional manifold (without boundary) in $\RR^n$.
Recall that the {\bf entropy} of $M$ is
\begin{equation}\label{entropy-definition}
\begin{aligned}
e(M)
&=
\sup_{a\in \RR^n, \, \lambda>0} (4\pi \lambda)^{-m/2} \int_{x\in M} e^{- |x-a|^2/(4\lambda)}\,dx \\
&=
\sup_{\tilde M} \, (4\pi)^{-m/2} \int_{x\in \tilde M} e^{-|x|^2/4}\,dx,
\end{aligned}
\end{equation}
where the second supremum is over all surfaces $\tilde M$ obtained from $M$ by translating and dilating.

More generally, suppose $M$ is a Radon measure on $\RR^n$.  The $m$-dimensional entropy of $M$ is
\begin{equation}
\begin{aligned}
e(M)
&=
\sup_{a\in \RR^n, \, \lambda>0} (4\pi \lambda)^{-m/2} \int e^{- |x-a|^2/(4\lambda)}\,dM \\
&=
\sup_{\tilde M} \int (4\pi)^{-m/2} \int_{x\in \tilde M} e^{-|x|^2/4}\,d\tilde M.
\end{aligned}
\end{equation}
where the supremum is over all Radon measures $\tilde M$ obtained from $M$ by translation and $m$-dimensional
scaling.
(If $M$ is a Radon measure and $\lambda>0$, the $m$-dimensional rescaling of $M$ by $\lambda$ is the measure 
$\lambda_\#M$ obtained by pushing $M$ forward by $x\mapsto \lambda x$ and then multiplying by $\lambda^m$.)

If $t\in I\mapsto M(t)$ is a smooth mean curvature flow of properly immersed $m$-manifolds in $\RR^n$ or, more generally, if it is an $m$-dimensional Brakke flow in $\RR^n$, then, by Huisken's monotonicity formula, the entropy 
  $e(M(t))$ is a decreasing function of $t\in I$.
  
Now suppose that $\Gamma$ is a smooth, properly embedded $(m-1)$-dimensional manifold in $\RR^n$, and suppose that $M$ is an $m$-manifold with boundary $\Gamma$.  
If $v\in \RR^n$, we let $[M,\Gamma,v]$ be the piecewise-smooth manfold obtained by attaching the exterior cone
\[
    \{ v + s(x-v): v\ge 1\}
\]
to $M$.
More generally, if $M$ is any Radon measure on $\RR^n$, we let $[M,\Gamma,v]$ be the Radon measure
\[
  M + E_{\Gamma,v}.
\]
We define
\newcommand{\gauss}{\Theta_\textnormal{gauss}}
\begin{equation}
\gauss(M,\Gamma,v,r)
=
\int  \frac{\exp( - |x-v|^2/(4r^2) )}{(4\pi r^2)^{m/2}}\,d[M,\Gamma,v]x   + \frac12 1_{\Gamma}(x),
\end{equation}
where
\[
1_{\Gamma}(x)
=
\begin{cases}
1 &\text{if $x\in \Gamma$, and} \\
0 &\text{if $x\notin \Gamma$}.
\end{cases}
\]
The term $\frac12 1_{\Gamma}(x)$ is necessary to make $\gauss(M,\Gamma,v,r)$ continuous as a function of $v$.
(To see that $\gauss(M,\Gamma,v,r)$ depends continuously on $v$, note that if $v_i\notin \Gamma$ converges to $v\in \Gamma$,
then, after passing to a subsequence, $E_{\Gamma,v_i}$ converges to $E_{\Gamma,v}$ together with a halfplane bounded by $\Tan(\Gamma,v)$.)

We define the $m$-dimensional entropy $e(M;\Gamma)$ of the pair $(M,\Gamma)$ to be
\begin{equation}\label{entropy-pair-definition}
\begin{aligned}
e(M;\Gamma)
&=
\sup_{a\in \RR^n, r>0} \gauss(M,\Gamma,a,r)  \\
&=
\sup_{(\tilde M, \tilde \Gamma)} \gauss(\tilde M, \tilde \Gamma, 0, 1),
\end{aligned}
\end{equation}
where the second supremum is over all pairs $(\tilde M, \tilde \Gamma)$ obtained from $(M,\Gamma)$ by
translation and dilation.  Note that $e(M;\emptyset)=e(M)$.

\begin{theorem}\label{extended-entropy-monotonicity}
Let $t\in I\mapsto M(t)$ be an $m$-dimensional Brakke flow in $\RR^n$ with boundary $\Gamma$.  Then 
\[
  t\in I \mapsto e(M(t);\Gamma)
\]
is a decreasing function of $t$.
\end{theorem}

\begin{proof}
This is an immediate consequence of the Monotonicity Theorem~\ref{extended-monotonicity}.
\end{proof}

\begin{theorem}
Suppose that $M$ is a shrinker, i.e., that
\[
   t\in (-\infty,0) \mapsto |t|^{1/2}_\# M
\]
is a mean curvature flow.  Then, in the definition~\eqref{entropy-definition} of entropy, the supremum is attained for $a=0$ and $\lambda=1$.
Furthermore, the density of $M$ at every point $p\in M$ is $\le e(M)$, and if equality holds at any point,
 then $M$ is a cone about that point.
\end{theorem}

See for example~\cite{colding-minicozzi-generic}*{lemma~7.01}, or, for the first assertion, the proof of Theorem~\ref{entropy-attained} below.
(In~\cite{colding-minicozzi-generic}*{lemma~7.01}, the shrinker is assumed to have polynomial area growth, but that assumption is not necessary since every shrinker has polynomial 
area growth.  See for example~\cite{brendle-zero}*{Proposition~10}.)

\begin{theorem}\label{entropy-attained}
Suppose that $M$ is an $m$-dimensional shrinker in $\RR^{n}$ whose boundary is an $(m-1)$-dimensional
linear subspace $L$.
Then in the definition~\eqref{entropy-pair-definition} of $e(M;L)$, the supremum is attained for $a=0$ and $r=1$.
Furthermore, the extended density of $\Sigma$ at each point is $\le e(M,L)$, and if equality holds at any point, then
$\Sigma$ is a cone.
\end{theorem}

\begin{proof}
The proof is essentially the same as in the boundaryless case. 
For the reader's convenience, we give the proof of the first assertion: that the supremum in~\eqref{entropy-pair-definition}
is attained for $a=0$ and $r=1$.
Let
\[
M(t) 
=
\begin{cases} 
|t|^{1/2}_\#M  &(t\le 0), \\
0 &(t>0).
\end{cases}
\]
Thus $M(\cdot)$ is a mean curvature flow with boundary $L$.
 Let
 \[
   \Theta = \Theta_\textnormal{gauss}(M,\Gamma,0,1) = \Theta_\textnormal{gauss}(M(-1),\Gamma,0,1)
 \]
 Trivially $e(M;L)\ge \Theta$.  We must show that $e(M,L)\le \Theta$.
By self-similarity, 
\[
  \Theta = \Theta_\textnormal{gauss}(M(-r^2), \Gamma, 0, r)
\]
for all $r>0$.
By monotonicity (Theorem~\ref{extended-monotonicity}),
\[
   \Theta_\textnormal{gauss}(M(t-r^2), \Gamma, a, r)
\]
is an increasing function of $r$ for every spacetime point $(a,t)$.  From polynomial area growth of $M$~\cite{brendle-zero}*{Proposition~10},
one easily checks that
\[
  \lim_{r\to\infty} \Theta_\textnormal{gauss}(M(t-r^2,\Gamma,a,r) = \lim_{r\to\infty}\Theta_\textnormal{gauss}(M(-r^2),\Gamma,0,r) = \Theta.
\]
Thus 
\[
  \Theta \ge  \Theta_\textnormal{gauss}(M(t-r^2,\Gamma,a,r)
\]
for all $r>0$.
Since $t$ is arbitrary, we see that for all $\rho$, $a$, and $r$,
\begin{align*}
\Theta 
&\ge  \Theta_\textnormal{Gauss}(M(-\rho^2, \Gamma,a,r)  \\
&= \Theta( \rho_\#M, \Gamma, a, r) \\
&= \Theta( M, \Gamma, a/\rho, r/\rho).
\end{align*}
Since $a$ and $\rho$ are arbitrary, we see that
\[
  \Theta \ge \sup_{p, R} \Theta(M, \Gamma, p, R) = e(M;L).
\]
\end{proof}

\section{Entropy and Maximal Density Ratio}\label{maximal-density-section}

\newcommand{\mdr}{\operatorname{mdr}}

Entropy is closely related to a more geometrically intuitive notion, namely the maximal density ratio.  In particular, 
for any surface (or Radon measure) the ratio of the two quantities is bounded above and below.

The ($m$-dimensional) {\bf maximal density ratio} $\mdr(M)$ of a surface $M$  
  in $\RR^n$ is defined to be
\[
\mdr(M)
:=
\sup_{x\in \RR^n, \, r>0} \, \frac{\Hh^m(M\cap \BB(x,r))}{\omega_mr^m},
\]
where $\omega_m$ is the volume of the unit ball in $\RR^m$.
Of course $\mdr(M)$ is invariant under rigid motions and scaling.
More generally, if $M$ is a Radon measure on $\RR^n$, the the $m$-dimensional maximal density ratio of $M$
is
\begin{equation}
\mdr(M)
:=
\sup_{a\in \RR^n, \, r>0} \, \frac{M\BB(x,r)}{\omega_mr^m}.
\end{equation}

\begin{theorem}\label{entropy-mdr-theorem}
$\mdr(M) \ge e(M) \ge c_m \mdr(M)$ for some $c_m>0$.
\end{theorem}

\begin{proof}  Let us assume that $M$ is an $m$-manifold; except for notation, the same proof works for Radon measures.  Let $\tilde M$ be a surface obtained from $M$ by translating and scaling.  Let
\[
   V(r) = \Hh^m(\tilde M\cap \BB(0,r))
\]
and $\theta=\mdr(M)$.  Thus $V(r)\le \omega_m\theta r^m$, so
\begin{equation}\label{parts}
\begin{aligned}
\int_{ \tilde M\cap\{|x|\le R\}} e^{-|x|^2/4}\,dx
&=
\int_{r=0}^R  e^{-r^2/4} V'(r)\,dr  \\
&=
 e^{-R^2/4} V(R)
-  \int_{r=0}^R  \frac{d}{dr}\left(e^{-r^2/4}\right) V(r)\,dr  \\
&=
 e^{-R^2/4} V(R) +  \int_{r=0}^R \frac12 r e^{-r^2/4} V(r)\,dr  \\
&\ge
 e^{-R^2/4} V(R).
%\theta \, 
 %\frac{\omega_m}2 \int_{r=0}^\infty  r^{1-m} e^{\frac{-r^2}{4}}\,dr  \\
\end{aligned}
\end{equation}
To prove the lower bound for $e(M)$, let $0<\eta<1$  We may assume by translating and scaling that $V(1)\ge \eta\omega_m \mdr(M)$.
Thus by~\eqref{parts} with $R=1$,
\[
(4\pi)^{m/2} e(M) \ge e^{-1/4}V(1) \ge e^{-1/4}(\eta\,\omega_m) \mdr(M).
\]
Since this holds for all $\eta\in (0,1)$ it also holds for $\eta=1$.  Thus
\[
e(M) \ge (4\pi)^{-m/2}e^{-1/4}\omega_m \mdr(M).
\]

To prove the upper bound for $e(M)$, we may assume that $\theta:=\mdr(M)$ is finite.
Now $V(r)\le \omega_m\theta r^m$, so from~\eqref{parts} we see that
\begin{equation}\label{lionel}
\begin{aligned}
\int_{ \tilde M\cap\{|x|\le R\}} e^{-|x|^2/4}\,dx
&=
e^{-R^2/4} V(R) +  \int_{r=0}^R \frac12 r e^{-r^2/4} V(r)\,dr  \\
&\le
e^{-R^2/4} \omega_m\theta R^m +  \int_{r=0}^R \frac12 r e^{-r^2/4} 
 \omega_m \theta r^m \,dr.
\end{aligned}
\end{equation}
Letting $R\to\infty$ and multiplying by $(4\pi)^{-m/2}$ gives

\begin{equation}\label{integral-is-one}
\begin{aligned}
(4\pi)^{-m/2} \int_{ \tilde M} e^{-|x|^2/4}\,dx
&\le
\theta   \int_{r=0}^R (4\pi)^{-m/2} (\omega_m/2) r^{1+m} e^{-r^2/4}  \,dr   
\\
&=
\theta.
\end{aligned}
\end{equation}

Taking the supremum over all $\tilde M$ gives $e(M)\le\theta=\mdr(M)$.

(Here is one way to see that the definite integral on the right side of~\eqref{integral-is-one} is $1$.  Note that if $\tilde M$ is an $m$-plane through the origin, then $V(r)\equiv \omega_mr^m$, so the inequalities in~\eqref{lionel}
 and in~\eqref{integral-is-one} become
equalities.  Furthermore, in this case, the left side of~\eqref{integral-is-one} is $1$ and $\theta=1$, so
 the definite integral on the right side of~\eqref{integral-is-one} must be equal to $1$.)
\end{proof}

\begin{remark}\label{weighted-remark}
{\rm
The proof shows that if $M$ is a Radon measure on $\RR^n$, then $M\rho$
is a weighted average of the density ratios $(M\BB(0,r))/(\omega_m r^m)$ over $r\in (0,\infty)$.  More generally,
$M\psi_{v,a}$ is a weighted average of $(M\BB(v,r))/(\omega_m r^m)$ over $r\in (0,\infty)$.
}
\end{remark}

Now we turn to manifolds (or varieties) with boundary.
Suppose $\Gamma$ is a smooth $(m-1)$-manifold without boundary in $\RR^n$.
If $M$ is an $m$-dimensional manifold with boundary $\Gamma$, we let
\[
\Theta(M,\Gamma,a,r)
=
\frac{\area((M\cup E_{\Gamma,a})\cap \BB(a,r))}{\omega_m r^m}  + \frac12 1_{\Gamma}(x),
\]
where $E_{\Gamma,v}$ is the exterior cone over $\Gamma$ with vertex $v$
 (as in Section~\ref{euclidean-monotonicity-section}).

More generally, if $M$ is a Radon measure on $\RR^n$, we let
\[
\Theta(M,\Gamma,a,r)
=
\frac{(M + E_{\Gamma,a}) \BB(a,r)}{\omega_m r^m}  + \frac12 1_{\Gamma}(x),
\]
where $E_{\Gamma,a}$ is the Radon measure associated to the exterior cone.

Given a point $a\in \RR^n$, note that $(M + E_{\Gamma,a})\partial \BB(a,r)=0$ for almost all $r$, and that for such $r$,
the function $\Theta(M,\Gamma,\cdot, r)$ is continuous at $a$.

We also let
\[
\mdr(M;\Gamma) = \sup_{a\in \RR^n, \, r>0} \Theta(M,\Gamma,a,r).
\]

\begin{theorem}\label{entropy-mdr-boundary-theorem}
Suppose that $\Gamma$ is a smooth, property embedded $(m-1)$-dimensional submanifold of $\RR^n$ and that
$M$ is a Radon measure on $\RR^n$.  Then
\[
 \mdr(M;\Gamma) \ge e(M;\Gamma)  \ge c_m \mdr(M;\Gamma).
 \]
 \end{theorem}
 
 The proof is essentially identical to the proof of Theorem~\ref{entropy-mdr-theorem}.

\begin{remark}\label{pointed-remark}
{\rm
In the definition~\eqref{entropy-pair-definition} of $e(M;\Gamma)$, if we fix the point~$v$ and take the supremum over $r>0$, we get an entropy-like quantity $e(M;\Gamma)(v)$ that depends on $v$:
\[
  e(M;\Gamma)(v) := \sup_{r>0} \gauss(M,\Gamma,v,r).
\]
Likewise, we can define a fixed-center-point version of the maximal density ratio:
\[
  \mdr(M;\Gamma)(a) := \sup_{r>0}\Theta(M,\Gamma,a,r).
\]
Of course $e(M;\Gamma)=\sup_v e(M;\Gamma)(v)$ and $\mdr(M;\Gamma)=\sup_a \mdr(M;\Gamma)(a)$.
Furthermore, for each $v\in \RR^n$, 
\[
   \mdr(M;\Gamma)(v) \ge e(M;\Gamma)(v) \ge c_m \mdr(M;\Gamma)(v),
\]
and if $t\in I\mapsto M(t)$ is an $m$-dimensional Brakke flow with boundary $\Gamma$, then $e(M(t);\Gamma)(v)$ is a
decreasing function of $t$.   The proofs are identical to the proofs of Theorems~\ref{entropy-mdr-boundary-theorem} 
  and~\ref{extended-entropy-monotonicity}.
  }
\end{remark}

\section{Compactness Theorems}\label{compactness-section}

\begin{theorem}\label{flow-compactness}
Suppose for $i=1,2, \dots$ that 
\[
  t \in [0,T]\mapsto M_i(t)
\]
is an integral Brakke flow in $U$ with boundary $\Gamma_i$.
Suppose also that the $\Gamma_i$ converge in $C^1$ to a smooth, properly embedded $(m-1)$-dimensional
submanifold $\Gamma$ of $U$, and that 
\[
   c_K:= \sup_i \sup_{t\in [0,T]}M_i(t)K <\infty
\]
for each compact subset $K$ of $U$.
Then there is a subsequence $i(j)$ such that for each $t\in [0,T]$, $M_{i(j)}(t)$ converges to a Radon measure $M(t)$.
\end{theorem}

\begin{proof}
Let $\Ff$ be a countable collection of $C^2$, nonnegative, compactly supported functions on $U$ such that
the linear span of $\Ff$ is dense in the space of all continuous, compactly supported functions.
By Corollary~\ref{measure-increasing-corollary}, for each $u\in \Ff$, the function
\begin{equation}\label{the-functions}
   t\in [0,T]\mapsto M(t)u - c_{\spt u} (\max |\nabla^2u|) t
\end{equation}
is non-increasing.  Each such function is also bounded.
Hence by passing to a subsequence, we can assume that each of the functions~\eqref{the-functions}
converges to a limit function.  Theorem~\ref{flow-compactness}
 follows immediately from the Riesz Representation Theorem.
\end{proof}

In the following theorem, we write $H_i(x,t)$ for $H(M_i(t),x)$ and $\nu_i(x,t)$ for $\nu(M_i(t),x)$.

\begin{theorem}\label{main-flow-compactness-theorem}
Suppose
\begin{enumerate}
\item $\Gamma_i$ $(i\in \NN)$ and $\Gamma$ are smooth, $(m-1)$-dimensional submanifolds of $U$,
                   and the $\Gamma_i$ converge smoothly to $\Gamma$.
\item For $i\in \NN$, $t\in [0,T]\mapsto M_i(t)$ is an $m$-dimensional integral Brakke flow with boundary $\Gamma_i$.
\item For each $t$, $M_i(t)$ converges to a Radon measure $M(t)$.
\end{enumerate}
Then $t\mapsto M(t)$ is an integral Brakke Flow with boundary $\Gamma$,
and
\begin{equation}\label{local-area-bound}
   c_K:= \sup_i \sup_{t\in [0,T]}M_i(t)K <\infty.
\end{equation}
Furthermore, there is a continuous, everywhere positive function $\phi:U\to\RR$ with
the following property.
For almost every $t$, there is a subsequence $i(j)$ such that:
\begin{equation}\label{H-squared-ok}
\sup_j\int \phi |H(M_{i(j)}(t),\cdot)|^2\,dM_{i(j)}(t) < \infty,
\end{equation}
\begin{equation}\label{just-H-bound}
    \sup_j\left( \int_K |H_{i(j)}(t,\cdot)|\,dM_{i(j)}(t) + \beta(M_{i(j)}(t))K \right) < \infty \quad\text{for all $K\subset\subset U$},
\end{equation}
\begin{equation}\label{varifolds-ok}
      \Var(M_{i(j)}(t)) \to \Var(M(t)),  
\end{equation}
\begin{equation}\label{h-ok}
\begin{aligned}
 \int H_{i(j)}(x)\cdot X(x,  & \Tan(M_{i(j)},x))  \,dM_i(t) 
 \\
 &\to
 \int H(x)\cdot X(x,\Tan(M,x))\,dM(t), 
\end{aligned}
\end{equation}
\begin{equation}\label{nu-ok}
\begin{aligned}
 \int \nu_{i(j)}(x)\cdot Y(x,&\Tan(\Gamma_{i(j)},x))\,d\Gamma_{i(j)}(x)
 \\
 &\to
 \int \nu(x)\cdot Y(x,\Tan(\Gamma,x))\,d\Gamma(x)
 \end{aligned}
 \end{equation}
 for every $X\in \Xx_m(U)$ and
for every $Y\in \Xx_{m-1}(U)$.
\end{theorem}

\begin{remark}\label{new-h-remark}
{\rm
Even for Brakke flows without boundary, the fact that $X$ in~\eqref{h-ok} can depend
on $x$ and $\Tan(M_{i(j)},x)$ (rather than just on $x$) seems to be new.
}
\end{remark}

\begin{proof}
The finiteness of $c_K$ follows immediately from Theorem~\ref{trace-theorem}.
By Theorem~\ref{H-squared-control}, 
\begin{equation*}\label{d-bound}
   d_K:= \sup_i \int_0^T\int_K |H_i|^2\,dM_i(t)\,dt < \infty.
\end{equation*}
It follows that there is a continuous, everywhere positive function $\phi:U\to \RR$ such that
\begin{equation}\label{cap-C}
  C:= \sup_i \int_0^T \int\phi |H_i|^2\,dM_i(t)\,dt < \infty.
\end{equation}
By Fatou's Lemma,
\[
   \int_0^T \left(\liminf_i \int \phi |H_i|^2\,dM_i(t)\right)\,dt < \infty,
\]
so for almost every $t$, 
\[
    \liminf_i\int \phi |H_i|^2\,dM_i(t) < \infty.
\]
For every such $t$, there is a subsequence $i(j)$ such that
\[
  \sup_j  \int \phi |H_{i(j)}|^2\,dM_{i(j)}(t) < \infty.  %= \liminf\int \phi |H_i|^2\,dM_i(t) < \infty.
\]
By the Varifold Closure Theorem~\ref{varifold-closure-theorem}, 
$M(t)\in \Vv_m(U,\Gamma)$, and \eqref{just-H-bound}, 
\eqref{varifolds-ok},~\eqref{h-ok}, and~\eqref{nu-ok} hold.

It remains only to show that $t\mapsto M(t)$ is an integral Brakke flow with boundary $\Gamma$.
Let $u: U\times[0,T]\to \RR$ be a nonnegative, compactly supported, $C^2$ function. 
For each $i$,
\begin{align*}
&(M_iu)(a) - (M_iu)(b)
\\
&\quad\ge
\int_a^b\int \left( u|H_i|^2 - H_i\cdot\nabla u - \pdf{u}t \right)\,dM_i(t)\,dt
\\
&\quad =
\int_a^b\int \left( \left|u^{1/2}H_i - \frac12u^{-1/2}\nabla u\right|^2 - \frac14\frac{|\nabla u|^2}{u} - \pdf{u}t \right)\,dM_i(t)\,dt
\end{align*}
Therefore by~\eqref{cap-C},
\begin{align*}
&(M_iu)(a) - (M_iu)(b) + \eps\, C
\\
&\quad\ge
\int_a^b\int\left( \left|u^{1/2}H_i - \frac12u^{-1/2}\nabla u\right|^2 
                           + \eps\phi|H_i|^2 - \frac14\frac{|\nabla u|^2}{u} - \pdf{u}t \right)\,dM_i(t)\,dt.
\end{align*}
Letting $i\to\infty$ gives, by Fatou's Lemma,
\begin{equation}\label{after-fatou}
\begin{aligned}
&(Mu)(a) - (Mu)(b) + \eps\, C
\\
&\quad\ge
\int_a^b\lambda_\eps(t)\,dt
- \int_a^b\int \left( \frac14\frac{|\nabla u|^2}{u} + \pdf{u}t \right)\,dM(t)\,dt,
\end{aligned}
\end{equation}
where
\[
 \lambda_\eps(t) = \liminf_i \int\left( \left| u^{1/2}H_i - \frac12u^{-1/2}\nabla u \right|^2 + \eps\phi|H_i|^2\right)\,dM_i(t).
\]
For each $t$ with $\lambda_\eps(t)<\infty$, there is a subsequence $i(j)$ such that
\[
  \int \left( \left|u^{1/2}H_{i(j)} - \frac12u^{-1/2}\nabla u\right|^2 + \eps\phi|H_{i(j)}|^2 \right)\,M_{i(j)}(t) \to \lambda_\eps(t).
\]
For such $t$, we have (as above), $\Var(M_{i(j)}t))\to \Var(M(t))$ and
\[
 \int H_{i(j)}\cdot X\,dM_{i(j)} \to \int H\cdot X\,dM(t),
\]
for all $X\in \Xx(U)$.
Consequently,
\[
  \int \left(u^{1/2}H_{i(j)} - \frac12u^{-1/2}\nabla u\right)\cdot X\,dM_{i(j)}(t) 
  \to
  \int \left(u^{1/2}H - \frac12u^{-1/2}\nabla u\right)\cdot X\,dM(t).
\]
Since this holds for all $X\in \Xx(U)$,
\begin{align*}
  \int \left| u^{1/2}H- \frac12u^{-1/2}\nabla u \right|^2\,dM(t)
  &\le
  \liminf 
  \int \left| u^{1/2}H_{i(j)}- \frac12u^{-1/2}\nabla u \right|^2\,dM_{i(j)}(t)
  \\
  &\le \lambda_\eps(t)
\end{align*}
Substituting this into~\eqref{after-fatou} and letting $\eps\to 0$ gives
\begin{align*}
(Mu)(a) - (Mu)(b)
&\ge
\int_a^b\ \int \left|u^{1/2}H- \frac12u^{-1/2}\nabla u\right|^2\,dM(t)\,dt
\\
&\qquad - \int_a^b\int \left( \frac14\frac{|\nabla u|^2}{u} + \pdf{u}t \right)\,dM(t)\,dt
\\
&=
\int_a^b \int \left( u|H|^2 - H\cdot \nabla u - \pdf{u}t \right) \,dM(t)\,dt
\end{align*}
\end{proof}

\section{Tangent Flows}\label{tangent-flow-section}
Consider an integral Brakke flow $t\in I\mapsto M(t)$ in $N$ with boundary $\Gamma$.
As in~\S\ref{monotonicity-section}, we isometrically embed $N$ in a Euclidean space $\RR^d$.
We now discuss tangent flows at a spacetime point $(p_0,t_0)$.
By making a spacetime translation, it suffices to consider the case $(p_0,t_0)=(0,0)$.

For $\lambda>0$, let $M^\lambda$ be the result of applying the parabolic dilation
\[
     \Dd_\lambda: (x,t)\mapsto (\lambda x, \lambda^2t)
\]
to the flow
\[
   t\in I\cap (-\infty,0)\mapsto M(t).
\]
Just as for mean curvature flow without boundary,  monotonicity 
together with the compactness and closure theorems in \S\ref{compactness-section} implies existence of tangent flows:
for every sequence $\lambda(i)\to\infty$, there is a subsequence $\lambda(i(j))$ such that
the flows $M^{\lambda(i(j))}(\cdot)$ converge to 
 a flow $M'(\cdot): t\in (-\infty,0]\mapsto M'(t)$.
If $0\notin \Gamma$, it is an integral Brakke flow in the Euclidean space $\Tan(N,0)$.  If $0\in\Gamma$, it is an integral Brakke flow in $\Tan(N,0)$
with boundary $\Tan(\Gamma,0)$.  In either case, the tangent flow $M'(\cdot)$ is self-similar: it is 
invariant under parabolic dilations $\Dd_\lambda$ with
$\lambda>0$.

%Let $M:t\in (-\infty,0)\to M(t)$ be a nonzero self-similar integral brakke flow (without boundary).

\begin{definition}
We say that an integral Brakke flow $t\in I\mapsto M(t)$ with boundary $\Gamma$
is {\bf unit-regular} provided the following holds:
\begin{quote}
For each $p\in N$ and $t\in I$, if one of the tangent flows at $(p,t)$
is a multiplicity-$1$ plane or halfplane, 
then the flow $M(\cdot)$ is fully smooth in a spacetime 
neighborhood of $(p,t)$.  
\end{quote}
Equivalently,
\begin{quote}
For each $p\in N$ and $t\in I$, 
if the extended Gauss density $\Theta_e(M(\cdot),(p,t))$ is $1$, 
 then the flow $M(\cdot)$ is fully smooth in a spacetime 
neighborhood of $(p,t)$.
\end{quote}
\end{definition}

(See Definition~\ref{gauss-density-definition} for extended Gauss density.)

Here ``fully smooth" means ``smooth and with no sudden vanishing".

Unit-regularity does not prevent sudden vanishing; it just implies that such vanishing can only occur where the multiplicity is greater than one.  For example, consider the Brakke flow that consists of a nonmoving plane with multiplicity $k$ for $t\le 0$ and that vanishes at time $0$.  The flow is not unit-regular if $k=1$, but it is (vacuously) 
 unit-regular if $k>1$. 

\begin{remark}
{\rm
If $p\notin\Gamma$ and if\, $\Theta_e(M,(p,t))=1$,
then $(p,t)$ is a backwardly $C^{1,\alpha}$-regular point of the flow
 by Brakke's Regularity Theorem~\cite{brakke} if the ambient space
is Euclidean or by the Kasai-Tonegawa~\cite{kasai-tonegawa} generalization of that theorem
 for general ambient manifolds, 
and consequently is a backwardly $C^\infty$-regular point by~\cite{tonegawa-higher}.   
Presumably the analogous theorems are true for $(p,t)$ with $p\in \Gamma$ 
and with $\Theta_e(M(\cdot),(p,t))=1$.
If so, then every integral Brakke flow with boundary would have the backward smoothness
 (but not necessarily the full smoothness)
described in the definition of unit-regularity.
However, none of those facts are required for this paper;  here, the simpler
local regularity theorems in~\cite{white-local} suffice.
}
\end{remark}

\section{Mod $2$ Flat Chains}

\newcommand{\Lmrec}{\mathcal{L}_\textnormal{$m$-rec}}

Let $\Lmrec(U,\ZZ^+)$  denote the space of functions on $U$ that take values in the nonnegative integers,
 that are locally $L^1$ with respect to Hausdorff $m$-dimensional measure on $U$, and that vanish except on a countable union of m-dimensional $C^1$-submanifolds of $U$. We identify functions that agree except on a set of Hausdorff
  $m$-dimensional measure zero.
  Let $\Lmrec(U, \ZZ_2)$ be the corresponding space with the nonnegative integers $\ZZ^+$
   replaced by $\ZZ_2$, the integers mod $2$.
The space $I\MM_m(U)$ (defined in \S\ref{notation-section})
    is naturally isomorphic to $\Lmrec(U, \ZZ^+)$: given any $M\in I\MM_m(U)$, 
    the corresponding function in $\Lmrec(U, \ZZ^+)$ is the density function  $\Theta(M,\cdot)$ given by
\[
      \Theta(M,x) = \lim_{r\to 0}\frac{M\BB(x,r)}{\omega_mr^m},
\]
where $\omega_m$ is the volume of the unit ball in $\RR^m$. 
In particular, this limit exists and is a nonnegative integer for $\Hh^m$-almost every $x\in U$.
Similarly, the space of $m$-dimensional 
rectifiable mod $2$ flat chains\footnote{As in~\cite{simon-gmt} and in~\cite{white-duke}, we do not require 
flat chains to have compact support.  In Federer's terminology~\cite{federer-book}, they would be called
``locally flat chains".  See the discussion in~\cite{white-duke}*{\S2.1}.}
 in $U$ is naturally isomorphic to $\Lmrec(U, \ZZ_2)$: 
given any such flat chain $A$, the corresponding function is the density function $\Theta(A, \cdot)$ given by
\[
   \Theta(A, x) := \lim_{r\to 0} \frac{\mu_A\BB(x, r)}{\omega_mr^m}
\] 
where $\mu_A$ is the Radon measure on $U$ determined by $A$. 
In particular, this limit exists and is $0$ or $1$ for $\Hh^m$-almost every $x \in U$.

The surjective homomorphism
\begin{align*}
&[\cdot]: \ZZ^+ \to \ZZ_2, \\
&k \mapsto  [k]
\end{align*}
determines a homomorphism from $\Lmrec(U, \ZZ^+)$ to $\Lmrec(U, \ZZ_2)$
 and thus also a homomorphism from the additive semigroup $I\MM_m(U)$
  to the additive group of $m$-dimensional rectifiable mod $2$ flat chains in $U$. 
If $M\in I\MM_m(U)$, we let $[M]$ denote the corresponding rectifiable mod $2$ flat chain. 
Thus $[M]$ is the unique rectifiable mod $2$ flat chain in $U$ such that
\[
   \Theta([M],x) = [\Theta(M,x)]
\]
for $\Hh^m$-almost every $x\in U$.

The following is Theorem 3.3 in~\cite{white-duke}: 

\begin{theorem}\label{duke-journal-theorem}
Suppose that $M_i$ $(i=1,2,\dots)$ and $M$ are Radon measures in $I\MM_m(U)$ with the following 
properties:
\begin{enumerate}
\item $\Var(M_i)\rightharpoonup \Var(M)$.
\item Each $M_i$ has bounded first variation, and
\[
   q_K := \sup_i \left( \int_K |H(M_i,\cdot)|\,dM_i + \beta(M_i)(K) \right) < \infty.
\]
\item The $\partial [M_i]$ converge (in the flat topology) to a mod $2$ flat chain $C$.
\end{enumerate}
Then the $[M_i]$ converge (in the flat topology) to $[M]$.  In particular, $\partial [M_i]=C$.
\end{theorem}

\section{Standard Brakke Flows and The Closure Theorem}\label{standard-section}

Suppose that $(M_i(\cdot),\Gamma_i)$ ($i\in \NN$) and $(M(\cdot),\Gamma)$
are integral Brakke flows with boundary in $U$ defined on a time interval $I$.
We say that the $(M_i(\cdot),\Gamma_i)$ {\bf converge} to $(M(\cdot),\Gamma)$
if $M_i(t)\rightharpoonup M(t)$ for each $t\in I$ and the $\Gamma_i$ converge smoothly to $\Gamma$.

\begin{theorem}[Closure Theorem]\label{closure-theorem}
Suppose $(M_i(\cdot),\Gamma_i)$ converges to $(M(\cdot),\Gamma)$,
where each $M_i:t\in[ 0,T]\to M_i(t)$ is an integral Brakke flow with boundary $\Gamma_i$.
\begin{enumerate}
\item\label{unit-regular-preserved} If the flows $M_i(\cdot)$ are unit-regular, then so is the flow $M(\cdot)$.
\item\label{mod-2-preserved} If $\partial [M_i(t)] = [\Gamma_i(t)]$ for almost every $t$, then
                 $\partial [M(t)] = [\Gamma(t)]$ for almost every $t$.
\end{enumerate}
\end{theorem}

\begin{proof}
See~\cite{schulze-white-triple}*{Theorem~4.2} for the proof of
 Assertion~\eqref{unit-regular-preserved}.  (The proof in~\cite{schulze-white-triple}
is for Brakke flows without boundary, but the same proof works for flows with boundary: the proof is based
on the local regularity theorems in~\cite{white-local}, which are stated and proved both with and without boundary.)

We now prove Assertion~\eqref{mod-2-preserved}.  
By Theorem~\ref{main-flow-compactness-theorem}, for almost every $t$, there is a subsequence $M_{i(j)}(t)$
such that
\begin{equation}\label{varifolds-ok-again}
      \Var(M_{i(j)}(t)) \to \Var(M(t))
\end{equation}
and
\begin{equation}\label{just-H-bound-again}
    \sup_j
    \left( \int_K |H_{i(j)}(t,\cdot)|\,dM_{i(j)}(t) + \beta(M_{i(j)}(t))K \right) < \infty 
    \quad\text{for all $K\subset\subset U$}.
\end{equation}

Since the $\Gamma_i$ converge smoothly to $\Gamma$, it follows that the $[\Gamma_i] (=\partial [M_i])$
converge to $[\Gamma]$ as mod $2$ flat chains.
By Theorem~\ref{duke-journal-theorem}, $[M_{i(j)}(t)] \to [M(t)]$ and $\partial [M(t)]=[\Gamma]$.
\end{proof}

\begin{definition}\label{standard-definition}
An integral Brakke flow $M(\cdot)$ with boundary $\Gamma$ is called {\bf standard}
if it has properties described in Theorem~\ref{closure-theorem}:
\begin{enumerate}
\item the flow is unit-regular, and
\item $\partial [M(t)]=[\Gamma]$ for almost every $t$.
\end{enumerate}
\end{definition}

\section{Existence}\label{existence-section}

\begin{theorem}\label{existence-theorem}
 Let $N$ be a smooth Riemannian manifold.  
If $N$ has nonempty boundary, we assume that the boundary is smooth and $m$-convex.
Let $\Gamma$ be smooth, properly embedded  $(m-1)$-dimensional manifold in $N$.
Let $M_0$ be a smoothly embedded $m$-dimensional manifold in $N$ with boundary $\Gamma$
and with finite area,
or, more generally, let $M_0$ be an $m$-rectifiable set of finite $m$-dimensional
measure such that $\partial [M_0]=[\Gamma]$.
Then there exists a standard Brakke flow 
\[
  t\in [0,\infty) \mapsto M(t)
\]
with boundary $\Gamma$ such that 
\[
    M(t) \rightharpoonup \Hh^m\llcorner M_0
\]
as $t\to 0$.

If $p\notin\Gamma$ and if $M_0$ is smooth in a neighborhood of $p$, then the flow is smooth
in a spacetime neighborhood of $(p,0)$.   If  $p\in \Gamma$ and if $M_0$ is $C^{1,\alpha}$ in a neighborhood
of $p$, then the flow is parabolically $C^1$ in a spacetime neighborhood of $(p,0)$.
\end{theorem}

Recall that $m$-convexity of $\partial N$ at a point $p$ is the condition that the sum of the smallest $m$ principal curvatures of $\partial N$
at $p$ with respect to the inward unit normal is nonnegative.  (Thus $1$-convexity is convexity, and if $\dim N=m+1$, then $m$-convexity is mean-convexity.)  
We say that $\partial N$ is strictly $m$-convex at $p\in \partial N$ if the sum of the smallest $m$-principal curvatures at $p$ is 
 greater than $0$.

\begin{comment}
\begin{theorem}\label{oriented-existence-theorem}
 Let $N$ be a smooth Riemannian manifold.  
If $N$ has nonempty boundary, we assume that the boundary is smooth and $m$-convex.
Let $(\Gamma,\gamma)$ be smooth, properly embedded, oriented  $(m-1)$-dimensional manifold in $N$.
Let $M_0$ be an oriented, smoothly embedded $m$-dimensional manifold in $N$ with oriented 
boundary $(\Gamma,\gamma)$.
More generally, $M_0$ can be a $m$-rectifiable set of finite $m$-dimensional Hausdorff measure
for which there is a Borel-measurable orientation of the tangent planes to $M_0$ such that the resulting
multiplicity-one current has boundary given by $(\Gamma,\gamma)$.
Then there exists a standard Brakke flow 
\[
  t\in [0,\infty) \mapsto M(t)
\]
with boundary $\Gamma$ such that 
\[
    M(t) \rightharpoonup \Hh^m\llcorner M_0
\]
as $t\to 0$, and such that $M(t)$ is boundary-compatible with $(\Gamma,\gamma)$ for almost every $t$.

If $p\notin\Gamma$ and if $M_0$ is smooth in a neighborhood of $p$, then the flow is smooth
in a spacetime neighborhood of $(p,0)$.   If  $p\in \Gamma$ and if $M_0$ is $C^{1,\alpha}$ in a neighborhood
of $p$, then the flow is parabolically $C^1$ in a spacetime neighborhood of $(p,0)$.
\end{theorem}
\end{comment}

\begin{proof}
Let us first prove the theorem assuming that $\partial N$ is strictly $m$-convex.

Let $\Gamma^* = (M_0\times \{0\})\cup (\Gamma\times[0,\infty)$.
Let $\lambda>0$.
Consider a mod $2$ flat chain $\Sigma^\lambda$ in $N\times\RR$ that minimizes
\[
     \int_{(x,z)\in N\times\RR} e^{-\lambda z}d\mu_{\Sigma^\lambda} (x,z)
\]
subject to $\partial \Sigma^\lambda = [\Gamma^*]$.
Let $M^\lambda$ be the associated Radon measure; thus $\Sigma^\lambda=[M^\lambda]$.

Note that $\Sigma^\lambda$ is mass-minimizing with respect to the Ilmanen metric $e^{-2\lambda z/m}(g+dz^2)$ 
(where $g$ is the metric on $N$.)
Thus $\spt(M^\lambda)\setminus \Gamma^*$ is smooth (with multiplicity $1$) away from a closed set
of Hausdorff dimension $\le m-1$ \cite{federer-short}.  

(Here is where strict $m$-convexity of $\partial N$ is used.  Note that strict $m$-convexity of $\partial N$ implies strict $(m+1)$-convexity
of $(\partial N)\times \RR$ with respect to the Ilmanen metric.  By the maximum principle in~\cite{white-max}, $\Sigma^\lambda$ cannot
touch $(\partial N)\times\RR$ at any interior point of $\Sigma^\lambda$.)

For $t\ge 0$, let $M^\lambda(t)$ be the portion of $M^\lambda- \lambda t(0,1)$ in $\tilde N=N\times(0,\infty)$.
Then
\[
    t\in [0,\infty) \mapsto M^\lambda(t)
\]
is an integral Brakke flow with boundary $\tilde \Gamma:=\Gamma\times (0,\infty)$ in $\tilde N$.
In fact, it is standard:
\begin{enumerate}
\item Because $M^\lambda$ is smooth almost everywhere, the mean curvature vector is orthogonal
to the surface almost everywhere.
\item Unit regularity follows from Allard's Regularity Theorem~\cite{allard-first-variation}*{\S8} 
 and Boundary Regularity Theorem~\cite{allard-boundary}*{\S4}
   applied to $M^\lambda$.
\item The mod $2$ boundary condition holds by construction.
\end{enumerate}

By~\cite{ilmanen-elliptic}*{5.1, 3.2(ii)}, the areas of the $\MM^\lambda(t)$ have a uniform upper bound on area
as $\lambda\to\infty$:
\[
   \area(\MM^\lambda_0 \cap \{a<z<a+b\}) \le (b + \lambda^{-1})\area(M_0)
\]
for any $a,b>0$, and thus
\[
   \area(\MM^\lambda(t)\llcorner\{0<z<b\}) \le (b+\lambda^{-1})\area(M_0).
\]

Consequently
 (by Theorems~\ref{flow-compactness}, \ref{main-flow-compactness-theorem}, and~\ref{closure-theorem}),
  the flows $t\in [0,\infty)\mapsto \MM^\lambda(t)$ converge 
as $\lambda\to\infty$ (after passing to a subsequence)
to a standard Brakke flow $\MM(\cdot)$ in $\tilde N$ with boundary $\tilde\Gamma$.

Furthermore, as in~\cite{ilmanen-elliptic},
\[
    \MM(0) = M_0\times (0,\infty)
\]
and
\[
   \MM(t) = M(t)\times (0,\infty)
\]
(except possibly for countably many $t$), where $M(t)$
is a Radon measure in $N$.

Since $\MM(\cdot)$ is a standard Brakke flow with boundary $\Gamma^*$ in $N^*$,
it follows that $t\in [0,\infty)\mapsto M(t)$ is a standard Brakke flow in $N$ with boundary $\Gamma$
and with
\[
    M(0)=M_0.
\]
(See~\cite{ilmanen-elliptic}*{8.9}.)

If $M_0$ is $C^{1,\alpha}$ in a neighborhood of a point $p$, then the flow is parabolically $C^1$
in a spacetime neighborhood of $(p,0)$ by~\cite{white-local}.
If $M_0$ is smooth in a neighborhood of a point $p\in N\setminus \Gamma$, 
then the flow is smooth in a spacetime neighborhood of $(p,0)$ by~\cite{schulze-white-triple}*{Corollary~A.3}.

This completes the proof assuming strict $m$-convexity.  For the general case, let $g_i$ be a sequence of smooth
metrics on $N$ such that the $g_i$ converge smoothly to the original metric and such that $\partial N$ is strictly $m$-convex
with respect to each $g_i$.  Then we get a suitable Brakke flow for each metric $g_i$.  
By Theorems~\ref{flow-compactness}, \ref{main-flow-compactness-theorem}, and~\ref{closure-theorem},
 a subsequence of those Brakke flows will converge to a suitable Brakke flow for the metric $g$.

(Theorems~\ref{flow-compactness}, \ref{main-flow-compactness-theorem}, and~\ref{closure-theorem} are
 stated for a fixed metric, but they hold, with the same proofs, for sequences of metrics.)
\end{proof}

\section{Self-Similar Flows}

\begin{theorem}[Shrinker Theorem]\label{shrinker-theorem}
Let $\Gamma$ be an $(m-1)$-dimensional linear subspace of $\RR^{m+1}$.
Suppose that 
\[
  t\in (-\infty,0)\mapsto S(t)
\]
is an $m$-dimensional integral Brakke flow in $\RR^{m+1}\setminus\Gamma$ that
is self-similar (i.e, invariant under parabolic dilations $\Dd_\lambda:(x,t)\mapsto (\lambda x, \lambda^2t)$ with $\lambda>0$).
Let $M:=S(-1)$, and suppose that $\spt(M)\setminus\Gamma$ is 
disjoint from some $m$-dimensional halfplane $P$ with boundary $\Gamma$.
Then $M$ is a sum of half-planes (each with boundary $\Gamma$) with multiplicities.
\end{theorem}

\begin{proof}
Let $\Pp$ be the set of halfplanes with boundary $\Gamma$ 
that are disjoint from $(\spt M)\setminus \Gamma$.
We claim that $\Pp$ is open.  To see this, suppose $P\in \Pp$.
By rotating, we can assume that $P$ is the halfplane $\{x_2=0,\, x_1\ge 0\}$.
Let $M^+$ and $S^+(t)$ be the portions of $M$ and $S(t)$ in the region
 $\{x_2\ge 0, \, x_1>0\}$.

Let $f:\BB^m(0,1)\cap \{ x_1>0\}\to \RR$ be a smooth, compactly supported, nonnegative
function that is $>0$ at some points. 
By multiplying by a small positive constant, we can assume that the graph of $f$ lies below $M^+$.
 Extend $f$ to $\BB^m(0,1)$ so that it is odd in $x_1$:
\[
    f(-x_1,x_2,\dots,x_m) =  - f(x_1,x_2,\dots,x_m).
\]
Now let 
\[
    u: \BB^m(0,1)\times [-1,\infty) \to \RR
\]
be the solution of the nonparametric MCF equation with 
\begin{align*}
&u(\cdot,-1)=f, \\
&u(\cdot,t)|\partial \BB^m\equiv 0.
\end{align*}
By the boundary maximum principle, $c:=\pdf{}{x_1}u(0,0)>0$.
Let $G(t,\eps)$ be the graph of $u(\cdot,t)-\eps$.
Let $\eps>0$.  By the maximum principle (Theorem~\ref{maximum-principle}),
$\spt(S^+(t))$ lies above $G(t,\eps)$ for all $t\in [-1,0]$.
Hence $\spt(S^+(t))$ lies in the closed region above $G(t)$ 
for $t\in [-1,0)$.
Equivalently, $\spt(M)$ lies in the closed region above $G^*(t):=|t|^{-1/2}G(t)$ for all $t\in [-1,0)$.  At $t\to 0$, $G^*(t)$ converges
to the plane $\{x: x_2=cx_1\}$.  Thus we see that the halfplane $P_\lambda:=\{x: x_2= \lambda x_1, \,x_1>0\}$
is in $\Pp$ for all $\lambda \in [0,c)$.  
Likewise there is a $c'<0$ such that the halfplane  $P_\lambda$ is disjoint from $M$ for all $\lambda\in (c',0]$.
This completes the proof of openess of $\Pp$.

Now let $P$ be a halfplane in the boundary of $\Pp$.  Then $\spt(M)$ touches $P$, so, by the strong maximum
principle, $M$ contains $P$.  (Note that $P$ and the varifold associated to $M$ are both stationary with respect to the shrinker metric,
so $M$ contains $P$ by the strong maximum principle in~\cite{solomon-white}.)  Now repeat the process with $S(t)$ replaced by
\[
    t \in (-\infty,0)\mapsto S(t) - (\Hh^m\llcorner P).
\]
The process must stop in finitely many steps, since otherwise $M$ would contain infinitely many
halfplanes and thus would not have locally finite area in $\RR^{m+1}\setminus \Gamma$.
\end{proof}

\begin{definition}\label{wedge-definition}
Consider two distinct $m$-dimensional linear subspaces $P$ and $P'$ of $\RR^{m+1}$.
The closure of a  component of $\RR^{m+1}\setminus (P\cup P')$ is called a {\bf wedge},
and $P\cap P'$ is the {\bf edge} of the wedge.
\end{definition}

\begin{corollary}\label{shrinker-no-boundary-corollary}
Suppose $t\in (-\infty,0)\mapsto M(t)$ is a nontrivial $m$-dimensional self-similar, integral Brakke flow (without boundary) in
$\RR^{m+1}$.  Then $\spt M(-1)$ is not contained in any wedge.
\end{corollary}

\begin{proof}
If $\spt M(-1)$ were contained in such a wedge $W$, then by
 Theorem~\ref{shrinker-theorem}
it would be a union of half-planes in $W$, which is impossible.
\end{proof}

\begin{corollary}\label{static-shrinker-corollary}
Let $V$ be a stationary integral $m$-varifold in $\RR^{m+1}\setminus \Gamma$ that is invariant
under positive dilations about $0$.  Suppose there is an open halfplane with boundary $\Gamma$
that is disjoint from the support of $V$.  Then $V$ is a sum of halfplanes with multiplicities.
\end{corollary}

This is the special case of Theorem~\ref{shrinker-theorem} when the shrinker is a minimal cone.

\begin{remark}
{\rm
Theorem~\ref{shrinker-theorem}  and Corollary~\ref{shrinker-no-boundary-corollary} remain true (with the same proofs) if the
hypothesis that $M(\cdot)$ is an integral Brakke flow is replaced
 by the hypothesis that $M(\cdot)$ is a Brakke flow such that the density of $M(-1)$ is $\ge a>0$ 
 almost everywhere (with respect to $M(-1)$).
Likewise Corollary~\ref{static-shrinker-corollary} remains true for any stationary varifold $V$
such that $\Theta(V,\cdot)\ge a > 0$ holds $\mu_V$ almost everywhere.
}
\end{remark}

%\newpage

\section{The Wedge Theorem and Boundary Regularity}

\begin{theorem}[Wedge Theorem]\label{wedge-theorem}
Suppose $W$ is a wedge (see Definition~\ref{wedge-definition}) in $\RR^{m+1}$ with edge $\Gamma$.
Suppose 
\[
   t\in (-\infty,0)\mapsto S(t)
\]
 is a self-similar, standard Brakke flow in $W$ with boundary $\Gamma$.
Then $S(\cdot)$ is a non-moving halfplane with multiplicity $1$.
\end{theorem}

\begin{proof}
Let $M=S(-1)$.
By Theorem~\ref{shrinker-theorem}, $M=\sum_{i=1}^k P_i$
where each $P_i$ is a multiplicity-one halfplane  with boundary $\Gamma$.
(The $P_i$ need not be distinct, since a halfplane in the support of $M$ is allowed to have any positive integer multiplicity.)
Since $\partial [M]=[\Gamma]$, $k$ is odd.

For each $i$, let $\nu_i$ be the unit vector in the plane of $P_i$ that is normal to $\Gamma$
and that points out from $P_i$.  For any smooth, compactly supported vectorfield $X$,
\begin{equation*}
\int \Div_{P_i}X\,dP_i = \int_\Gamma X\cdot \nu_i\,d\Hh^{m-1},
\end{equation*}
so
\begin{align*}
\int \Div_MX\,dM
&=
\int_\Gamma \left(\sum_{i=1}^k \nu_i\right) \cdot X \, d\Hh^{m-1}.
\end{align*}
Thus
\[
  \nu(M,\cdot) = \sum_{i=1}^k \nu_k.
\]
By definition of mean curvature flow with boundary, $|\nu(M,\cdot)|\le 1$.
Now we use the following elementary fact:
if $\nu'=\sum_{i=1}^k\nu_i$ where the $\nu_i=(\cos\theta_i,\sin\theta_i)$ are unit vectors 
with $|\theta_i|\le \theta < \pi/2$, then
\begin{equation}\label{trig}
  |\nu'|
  \ge
  \begin{cases}
  k\cos\theta &\text{if $k$ is even}, \\
  \sqrt{1 + (k^2-1)\cos^2\theta} &\text{if $k$ is odd}.
  \end{cases}
\end{equation}

(The inequalities~\eqref{trig} can be proved as follows.
Given $k$ and $\theta$, it is easy to show that at the mimimum of $|\nu'|$, each $\theta_i$ is $\pm\theta$.
If $k$ is even, the minimum is attained by having half of the $\theta_i$ equal to $\theta$ and the other
half equal to $-\theta$.  If $k=2j+1$ is odd, the minimum is attained when $j$ of the $\theta_i$ are equal
to $\theta$ and $j+1$ are equal to $-\theta$.)

In our case, the number of planes (counting multiplicity) in $M'$ is odd, so 
\[
  1 \ge |\nu'|\ge \sqrt{1 + (k^2-1)\cos^2\theta}.
\]
Therefore $k=1$.  
\end{proof}

As an immediate consequence of Theorem~\ref{wedge-theorem}, we have

\begin{theorem}\label{general-boundary-regularity-theorem}
Suppose $t\in [0,T]\mapsto M(t)$ is an $m$-dimensional standard mean curvature
flow with boundary $\Gamma$ in a smooth, $(m+1)$-dimensional Riemannian manifold.
If a tangent flow at $(p,t)$ is contained in a wedge, where $p\in \Gamma$ and $t>0$, 
then $(p,t)$ is a regular point of the flow.
\end{theorem}

\section{A Boundary Regularity Theorem}\label{boundary-regularity-section}

\begin{theorem}\label{main-boundary-regularity-theorem}
Suppose $N$ is a smooth, $(m+1)$-dimensional Riemannian manifold with smooth, weakly mean-convex boundary.
Suppose $\Gamma$ is a smooth, properly embedded $(m-1)$-dimensional submanifold of $N$.
Suppose
\[
   t\in [0,T]\mapsto M(t)
\]
is a standard Brakke flow in $N$ with boundary $\Gamma$.
If 
\begin{enumerate}%[\upshape (1)]
\item\label{reg-theorem-case-1} $\partial N$ is strictly mean convex, or if
\item\label{reg-theorem-case-2} $\spt(M(0))\cap \partial N =  \Gamma\cap \partial N$,
\end{enumerate}
then for every $p$ in $\Gamma\cap \partial N$ and for every $t\in (0,T]$, the spacetime point $(p,t)$
 is a regular point of the flow.
\end{theorem}

%HERE 

\begin{proof}
Note that if $\partial N$ is strictly mean convex, then by the maximum principle (\cite{ilmanen-elliptic}*{10.5} or
\cite{hershkovits-white-avoid}*{theorem~27}), 
\[
  \spt(M(t))\cap\partial N = \Gamma\cap\partial N
\]
for all $t>0$.   In other words, as soon as $t>0$, Hypothesis~\eqref{reg-theorem-case-2} holds.
Thus it suffices to prove Theorem~\ref{main-boundary-regularity-theorem}
 under Hypothesis~\eqref{reg-theorem-case-2}.

Since the result is local, it suffices to work in a small neighborhood of the point $p$.
Such a neighborhood is diffeomorphic to a halfspace, so we may assume that
\[
   N= \{x\in \RR^{m+1}: x_{m+1}\ge 0\}
\]
with some smooth Riemannian metric $g$. We may also assume that $p$ is the origin and that
the metric is Euclidean at the origin (i.e., that $g_{ij}(0)=\delta_{ij}$).
  In the rest of the proof, $\dist(\cdot,\cdot)$ refers to $g$-distance, but $\BB^{m+1}(a,r)$
  is the Euclidean ball of radius $r$ about $a$, and if $f$ is a function from a domain in $\RR^m$ to $[0,\infty)$
  (so that the graph lies in $N$), then expressions such as $\nabla f$ and $\|f\|_{C^2}$ are with
  respect to the Euclidean metric.
  
\begin{comment}
I'm modifying Lemma~\ref{flatish-disk-lemma} today (July 24).
\end{comment}

\begin{lemma}\label{flatish-disk-lemma}
For every $\eps>0$, there is a $\delta>0$ 
 with the following property.
If $a\in \partial N\cong\RR^m$, $|a|\le \delta$, $0<r\le \delta$,
and 
\[
f:\BB^m(a,r)\to [0,\infty)
\]
is a function with $\|f\|_{C^3}\le \eps$, then there 
is a solution
\[
    F: \BB^m(a,r)\times [0,\infty) \to \RR
\]
of the nonparametric mean curvature flow equation (with respect to the metric on $N$)
such that 
\begin{gather*}
 F(\cdot,0)=f(\cdot), \\
 F(\cdot,t)|\partial \BB^m(a,r) = f | \partial \BB^m(a,r), \\ 
\end{gather*}
and such that $|\nabla F|<\eps$ at all points.
\end{lemma}

\begin{proof}[Proof of lemma]
Let $\eps>0$.
If there were no suitable $\delta>0$, there would be a sequence of solutions 
\[
    F_i: \BB^m(a_i,r_i)\times [0,T_i]\to [0,\infty)
\]
of the nonparametric $g$-mean-curvature flow equation with $|a_i|\to 0$ and $r_i\to 0$
such that
\begin{equation}\label{c-3}
    \|F_i(\cdot,0)\|_{C^3} \to 0
\end{equation}
and such that 
\[
   \max |\nabla F_i| = |\nabla F_i(x_i,T_i)| = \eps.
\]
for some $x_i$.  Note that 
\[
   T:= \liminf \frac{T_i}{r_i^2} > 0
\]
by, for example, the local regularity theory in~\cite{white-local}. By passing to a subsequence, we can assume that $T_i/r_i^2\to T$.

Let $D_i(t)$ be the graph of $F_i(\cdot,t)$.  By~\eqref{c-3},
\[
   \frac{\area_g(D_i(0))}{\omega_mr_i^m} \to 1
\]
and thus
\begin{equation}\label{rather-isoperimetric}
    \sup_{t\in [0,T_i]} \frac{\area_g(D_i(t))}{\omega_mr_i^m} = \frac{\area_g(D_i(0))}{\omega_mr_i^m} \to 1
\end{equation}
as $i\to \infty$.  
Let $c_i$ be the average of $f$ over $\partial \BB(a_i,r_i)$.
Translating in space by $-(a_i, c_i)$ and in time by $-T_i$, dilating parabolically by $1/r_i$,
and passing to a subsequential limit gives a smooth solution
\[
    F: \BB^m(0,1) \times (-T,0] \to \RR
\]
of the Euclidean nonparametric mean curvature flow equation such that
\begin{gather*}
   F(\cdot,t) | \partial \BB^m(0,1) \equiv 0, \quad\text{and} \\
   \sup |\nabla F| = \sup |\nabla F(\cdot,0)| = \eps.
\end{gather*}
But by~\eqref{rather-isoperimetric}, 
 the area of the graph of $F(\cdot,t)$ is less than or equal to area of $\BB^m(0,1)$, and thus
$F\equiv 0$, a contradiction.  This proves the lemma.
\end{proof}

We now complete the proof of Theorem~\ref{main-boundary-regularity-theorem}.
Choose $a\in \Tan(\Gamma,0)^\perp \cap \partial N$ with $0<|a|<\delta$ 
(where $\delta$ is as in Lemma~\ref{flatish-disk-lemma} for $\eps=1$.)  We choose $a$ sufficiently close to $0$
that the Euclidean ball $\BB^{m+1}(a,|a|)$ intersects $\Gamma$ in the single
point $0$.  Let $2R=|a|$.

Let $f: \RR^m\to [0,\infty)$
be a smooth function such that: $f(a)>0$, $f$ is supported in the interior of $\BB^m(a,R)$, $\|f\|_{C^3}<\delta$,
and
\begin{equation}\label{starting-off-disjoint}
\text{$(\spt M(0))\setminus \Gamma$  lies in the set $\{x: x_{m+1}> f(x_1,\dots,x_m)\}$.}
\end{equation}

For $R \le r \le 2R$, let 
\[
F_r: \BB^m(a,r)\times [0,T] \to \RR
\]
be the solution of the nonparametric mean curvature flow equation (with respect
to the metric $g$) such that
\begin{align*}
&\text{$F_r(\cdot) = 0$ on $\partial \BB^m(a,r)$, and} \\
&\text{$F_r(\cdot,0)=f(\cdot)$ on $\BB^m(a,r)$.}
\end{align*}
By choice of $a$, 
\[
   |\nabla F_r(x,t)|\le 1
\]
for all $x\in \BB^m(a,r)$ and $t\in [0,T]$.  It follows that the graph $D_r(t)$ of $F_r(\cdot,t)$
is contained in the ball $\BB^{m+1}(a,r)$, and thus 
\begin{equation}\label{disjointed}
   D_r(t)\cap \Gamma =\emptyset \quad\text{for all $r\in [R, 2R)$ and $t\in [0,T]$}.
\end{equation}
By the strong maximum principle,
\[
    \text{$F_r(x,t)>0$ for all $x$ in the interior of $\BB(a,r)$ and all $t\in [0,T]$}
\]
and
\begin{equation}\label{hopf}
    \text{$|\nabla F_r(x,t)|\ne 0$ for all $x\in \partial \BB^m(a,r)$ and $t\in (0,T]$.}
\end{equation}
Let
\[
   Q_r(t) = \{x: 0< x^{m+1} < F_r(x_1,\dots,x_m,t)\}.
\]

We claim that 
\begin{equation}\label{disjoint-claim}
\text{If $R\le r < 2R$, then $\overline{Q_r(t)}$ is disjoint from $\spt M(t)$ for $t\in [0,T]$.}
\end{equation}
For suppose not. At the first time $t$ of contact, let $q$ be
a point in $\spt M(t)\cap \overline{Q_r(t)}$.  Then $q$ is in the graph $D_r(t)$ of $F_r(\cdot,t)$.
If $q$ were in $\partial D_r(t)$, then the tangent flow to $M(\cdot)$ at $(q,t)$ would be
contained in a wedge (by~\eqref{hopf}), which is impossible (see Corollary~\ref{shrinker-no-boundary-corollary}).
Thus
\begin{equation}\label{disjoint-redux}
    \partial D(\tau)\cap M(\tau) = \emptyset \quad\text{for all $\tau \in [0,t]$},
\end{equation}
and $q\in D(t)\setminus \partial D(t)$.
But this (together with~\eqref{starting-off-disjoint} and~\eqref{disjointed})
violates the maximum principle (Theorem~\ref{maximum-principle}).
This proves~\eqref{disjoint-claim}.

From~\eqref{disjoint-claim}, we see that
\begin{equation}\label{empty}
   \spt M(t) \cap Q_{2R}(t) = \emptyset
\end{equation}
since $\cup_{r<2R} Q_r(t)= Q_{2R}(t)$.

Now let $t\in (0,T]$.   By~\eqref{empty} and~\eqref{hopf}, any tangent flow to $M(\cdot)$
at $(p,t)$ must be contained in a wedge.  
Thus $(p,t)$ is a regular point of the flow $M(\cdot)$ by Theorem~\ref{general-boundary-regularity-theorem}.
\end{proof}

\section{Moving Boundaries}\label{moving-boundary-section}

Let $I$ be an interval in $\RR$.
We say that $\Gamma$ is 
a {\bf moving $(m-1)$-dimensional boundary} in $U\times I$ if $\Gamma$ is  a smooth, properly embedded, $m$-dimensional 
manifold-with-boundary in $ U\times I$ 
such that the boundary of $\Gamma$ is 
\[
     \Gamma\cap (U\times \partial I)
\]
and such that the time function $(x,t)\in \Gamma\mapsto t$ has no critical points on $\Gamma$.
For $t\in I$, we let $\Gamma(t)=\{x: (x,t)\in \Gamma\}$.  For $(x,t)\in \Gamma$, we let $\dot\Gamma(x,t)$ be the
{\bf normal velocity} of $\Gamma(t)$ at $x$: it is the unique vector $v\in \Tan(\Gamma(t),x)^\perp$ 
such that $(v,1)$
is tangent to $\Gamma$ at $(x,t)$.  Note that each $\Gamma(t)$ is a smooth, properly embedded $m$-dimensional submanifold of 
  $U$.

\begin{definition}\label{moving-boundary-mcf}
Let $\Gamma\subset U\times I$ be a moving $(m-1)$-dimensional boundary.
A {\bf Brakke flow with (moving) boundary $\Gamma$} is a Borel map
\[
   t\in I \mapsto M(t) \in \MM(U)
\]
such that
\begin{enumerate}[\upshape (1)]
\item For almost every $t\in I$, $M(t)$ is in $\Vv_m(U,\Gamma(t))$.
\item
If $[a,b]\subset I$ and $K\subset U$ is compact, then
\[
   \int_a^b\int_K (1+|H|^2)\,dM(t)\,dt < \infty.
\]   
\item\label{defining-inequality-moving-case}
If $[a,b]\subset I$ and if $u$ is a nonnegative, compactly supported $C^2$ function on $U\times[a,b]$,
 then
\begin{align*}
% &\int u(\cdot,a)\,dM(a) - \int u(\cdot,b)\,dM(b) \\
&(Mu)(a) - (Mu)(b) \\
 &\quad \ge
\int_a^b \int \left( u|H|^2 - H\cdot \nabla u - \pdf{u}{t} \right)\,dM(t)\,dt
-
\int_a^b\int  u \nu\cdot \dot\Gamma\,d\Gamma(t)\,dt.
\end{align*}
\end{enumerate}
\end{definition}

\begin{theorem}
Suppose $t\in I\mapsto M(t)$ is a Brakke flow with moving boundary $\Gamma$.
\begin{enumerate}[\upshape(1)]
\item
The defining inequality~\eqref{defining-inequality-moving-case} in Definition~\ref{moving-boundary-mcf}
 holds for every nonnegative, compactly supported,
Lipschitz function $u$ on $U$ that is $C^1$ on $\{u>0\}$.
\item
If 
\[
   f:U\times[a,b]\to \RR
\]
is a nonnegative, $C^2$ function with $\overline{\{f>0\}}$ compact, and if $u:= 1_{f\ge 0}f$, then
\begin{align*}
&(Mu)(a)-(Mu)(b)
\\
&\quad\ge
\int \left(u|H|^2 + \Div_M\nabla u  - \pdf{u}t \right)\,dM(t)\,dt + \int \nu\cdot(\nabla u - u\dot\Gamma)\,d\Gamma
\\
&\quad\ge
\int \left(u|H|^2 + \Tr_m(\nabla^2 u) - \pdf{u}t \right)\, dM(t)\, dt + \int \nu\cdot(\nabla u - u\dot\Gamma)\,d\Gamma.
\end{align*}
\item 
Suppose that
\begin{align*}
&\text{$\overline{\BB(x,R)}$ is a compact subset of $U$}, \\
&\text{$\dist(\cdot,x)^2$ is smooth on $\overline{\BB(x,R)}$,} \\
&\text{$\Tr_m(-\nabla^2(\dist(\cdot,x)^2))\ge -4m$ on $\BB(x,R)$, and} \\
&\text{$|\dot\Gamma|\le \delta$ on $\Gamma\cap(\BB(x,R)\times[0,T])$}.
\end{align*}
Let $u= (R^2-\dist(\cdot,x)^2- 4mt)^+$.
Then for $t\in [0,T]$,
\[
   (Mu)(t) \le  (Mu)(0) + t(R + \delta R^2)\Hh^{m-1}(\Gamma\cap \BB(x,R)).
\]
\item
Let $u$ be a  $C^2$, nonnegative, compactly supported function on $U$.
\begin{align*}
  \frac12 \int_a^b\int u |H|^2\,dM(t)\,dt 
  &\le  M(a)u- M(b)u  \\
  &\quad + (b-a)(\max|\nabla^2u|)K_{[a,b]} \\
  &\quad + (b-a) L_{[a,b]}
\end{align*}
for $[a,b]\subset I$, where 
\begin{align*}
   K_{[a,b]} &:= \sup_{t\in [a,b]}M(t)(\spt u), \\
   L_{[a,b]} &:= \sup_{t\in [a,b]} \int u |\dot\Gamma|\,d\Gamma(t).
\end{align*}
\end{enumerate}
\end{theorem}

The proofs are almost identical to the proofs of Proposition~\ref{more-general-u-proposition}, 
Corollary~\ref{trace-corollary}, Theorem~\ref{trace-theorem}, and Theorem~\ref{H-squared-control}.

\begin{theorem}\label{moving-boundary-monotonicity}
Let $U$ be an open subset of $\RR^d$ containing $\BB^d(0,1)$.
Let $N$ be a smooth, properly embedded submanifold of $U$.
Let $\Gamma$ be an $(m-1)$-dimensional moving boundary in $N\times [T_0, 0]$.
Let $t\in [T_0,0]\mapsto M(t)$ be an integral Brakke flow in $N$ with moving boundary $\Gamma$.
Suppose that
\[
    M(t)\BB(0,1)\le C
\]
for $t\in I$ and that the norm of the second fundamental form of $N$ is bounded by $A$.
Then for $a,b\in [T_0,0]$ with $a\le b< 0$, 
Then for $a,b\in I$ with $a\le b< 0$, 
\begin{equation}\label{moving-monotonicity-inequality}
\begin{aligned}
(M\hat\rho)(a) - (M\hat\rho)(b)
&\ge
\int_a^b \int \left| H_N - \frac{(\nabla^\perp\hat\rho)_N}{\hat\rho}\right|^2 \hat\rho\,dM(t)\,dt
\\
&\qquad
+ \int_a^b \int \nu_M\cdot (\nabla\hat\rho - \hat\rho\dot\Gamma)\,d\Gamma\,dt
\\
&\qquad
-  mA^2 \int_a^b \int \hat\rho\,dM(t)\,dt
\\
&\qquad
-
CK(b-a)
\end{aligned}
\end{equation}
where $C$ and $K$ are as in~\eqref{K-constant-monotonicity} 
and~\eqref{C-constant-monotonicity}.
Furthermore,
\begin{equation}\label{moving-monotone-quantity}
    e^{-mA^2t} \left( (M\hat\rho)(t) 
     + \int_{\tau=T_0}^t \int \nu_M\cdot (\nabla\hat\rho -\hat\rho\, \dot\Gamma)\, d\Gamma\,d\tau - CKt \right)
\end{equation}
is a decreasing function of $t$ for $t<0$ in $I$.
\end{theorem}

\begin{proof}
The proof is exactly like the proof of the  proof of the Monotonicity Theorem~\ref{monotonicity-theorem},
except that, starting with the right hand side of the first inequality in~\eqref{M-portion}, there is 
one additional extra term:
\[
  -  \int_a^b \int \hat\rho\, \dot\Gamma\cdot \nu\,d\Gamma\,dt.
\]
\end{proof}

\begin{corollary}\label{moving-monotonicity-corollary}
As $t\uparrow 0$, $(M\hat\rho)(t)$ converges to a finite limit.
\end{corollary}

\begin{proof}
A straightforward computation shows that
\begin{equation}\label{moving-finiteness}
 \int_{T_0}^0 \int | (\nabla\hat\rho - \hat\rho\,\dot\Gamma)_{\Gamma^\perp}|\,d\Gamma\,dt < \infty.
\end{equation}
The corollary now follows immediately from the monotonicity of~\eqref{moving-monotone-quantity}.
\end{proof}

\begin{comment}
\begin{remark}\label{half-remark}
The following slight refinement of~\eqref{moving-finiteness} is sometimes useful.
Let 
\[
\Phi(\Gamma,T)
=
\int_T^0 \int | (\nabla\hat\rho - \hat\rho \,\dot\Gamma)_{\Gamma^\perp}| \, d\Gamma \,dt
+ 
\begin{cases}
\frac12 &\text{if $(0,0)\in \Gamma$,} \\
0 &\text{if $(0,0)\notin \Gamma$}.
\end{cases}
\]
If $\Gamma_i$ and $\Gamma$ are moving $(m-1)$-dimensional boundaries in $U\times [T_0,0]$
(where $\BB^d(0,3)\subset U \subset \RR^d$) and if the $\Gamma_i$ converge smoothly to $\Gamma$,
then
\[
    \Phi(\Gamma_i,T) \to \Phi(\Gamma,T)
\]
that the $\Gamma_i$ converge smoothly to $\Gamma$.
\end{remark}

\marginpar{\color{blue} Check Remark~\ref{half-remark}.}
\end{comment}

The compactness and closure theorems, existence of tangent flows,
and the definition of standard flows are the exact analogs are
the corresponding theorems and definition for 
fixed boundaries (\S\ref{compactness-section}, \S\ref{tangent-flow-section}, \S\ref{standard-section})
so we will not state them.   
For example, in the statement of the Compactness Theorem~\ref{flow-compactness}, one simply replaces 
  ``Brakke flow with boundary $\Gamma_i$'' by
``Brakke flow with moving boundary $\Gamma_i$'' and 
 ``smooth, properly embedded $(m-1)$-dimensional submanifold $\Gamma$ of $U$'' 
 by
  ``moving $(m-1)$-dimensional boundary $\Gamma$ in $U\times [0,T]$''.
Likewise, in the statement of Theorem~\ref{main-flow-compactness-theorem}, one simply
replaces 
   ``smooth, $(m-1)$-dimensional submanifolds of $U$''
by
   ``moving $(m-1)$-dimensional  boundaries in $U\times [0,T]$''
 and 
    ``Brakke flow with boundary'' by ``Brakke flow with moving boundary''.

For those various theorems, the extra term arising from the motion of 
boundary is easy to control, so only trivial modifications of the proofs are required.

Just as for fixed boundaries, we have
(as an immediate consequence of the Wedge Theorem~\ref{wedge-theorem}),

\begin{theorem}\label{general-moving-boundary-regularity-theorem}
 Suppose that $M:[0,T]\mapsto M(t)$ is an
$m$-dimensional standard Brakke flow with moving boundary $\Gamma$
in an $(m+1)$-dimensional Riemannian manifold.  If $t>0$, if $p\in \Gamma(t)$,
and if a tangent flow at $(p,t)$ is contained in a wedge, then $(p,t)$ is a regular
point of the flow $M(\cdot)$.
\end{theorem}

Furthermore, the main boundary regularity theorem,
 Theorem~\ref{main-boundary-regularity-theorem}, continues to hold for moving boundaries.  The statement and the proof are almost exactly as in the non-moving versions, so we omit them.

Finally, the Existence Theorem~\ref{existence-theorem}
 in Section~\ref{existence-section}
 has an analog for moving boundaries:

\begin{theorem}\label{moving-existence-theorem}
 Let $N$ be a smooth Riemannian manifold.  
If $N$ has nonempty boundary, we assume that the boundary is smooth and $m$-convex.
Suppose that $\Gamma$ is a moving $(m-1)$-dimensional boundary
 in $N \times[0,\infty)$ such that 
 \begin{equation}\label{compactness-hypothesis}
    \textnormal{$\cup_{0\le t\le T} \Gamma(t)$ is compact for every $T<\infty$.}
 \end{equation}
Let $M_0$ be a smoothly embedded $m$-dimensional manifold in $N$
 with moving boundary $\Gamma(0)$
and with finite area,
or, more generally, let $M_0$ be an $m$-rectifiable set of finite $m$-dimensional
measure such that $\partial [M_0]=[\Gamma]$.
Then there exists a standard Brakke flow 
\[
  t\in [0,\infty) \mapsto M(t)
\]
with moving boundary $\Gamma$ such that 
\[
    M(t) \rightharpoonup \Hh^m\llcorner M_0
\]
as $t\to 0$.

If $p\notin\Gamma(0)$ and if $M_0$ is smooth in a neighborhood of $p$, then the flow is smooth
in a spacetime neighborhood of $(p,0)$.   If  $p\in \Gamma(0)$ and if $M_0$ is $C^{1,\alpha}$ in a neighborhood
of $p$, then the flow is parabolically $C^1$ in a spacetime neighborhood of $(p,0)$.
\end{theorem}

\begin{proof}[Sketch of proof] It is convenient to first prove the Theorem under the additional assumption
that there is a $T<\infty$ such that
\begin{equation}\label{eventually-still}
  \textnormal{$\Gamma(t) =  \Gamma(T)$ for all $t\ge T$}.
\end{equation}
With this assumption, the proof is almost identical to the proof of 
Theorem~\ref{existence-theorem}.  (The compactness hypothesis~\eqref{compactness-hypothesis} together with the 
assumption~\eqref{eventually-still} guarantee that the analog of $\Sigma^\lambda$ in the proof of 
Theorem~\ref{existence-theorem} has finite area with respect to the Ilmanen metric $e^{-2\lambda z/m}(g+dz^2)$ .)

For the general case, one takes a sequence $\Gamma^n$ of moving boundaries that 
\begin{align*}
\Gamma^n(t)&=\Gamma(t) \quad\text{for $0\le t\le n$}, \\
\Gamma^n(t) &= \Gamma^n(n+1) \quad\text{for $t\ge n+1$},
\end{align*}
and such that the compactness hypothesis~\eqref{compactness-hypothesis} holds for each $\Gamma^n$.  By the special case of the Theorem, there exists a
standard Brakke flow $\MM^n$ with moving boundary $\Gamma^n$ satisfying the conclusions of the Theorem.  By the moving boundary analogs of the Compactness Theorem~\ref{flow-compactness} and the
 Closure Theorem~\ref{main-flow-compactness-theorem},
  a subsequence of the flows $\MM^n$ will converge to a standard Brakke flow $\MM$ satisfying the conclusions of the Theorem.
\end{proof}

\section{Orientation}\label{orientation-section}

  In this section, we show that a certain weak notion of orientability
is preserved when taking weak limits of flows.

Suppose $\Gamma$ is an $(m-1)$-dimensional submanifold of $U$.
An {\bf orientation} of $\Gamma$ is a continuous function $\gamma$ that assigns
to each $x\in \Gamma$ a unit $(m-1)$-vector in $\Tan(\Gamma,x)$.
The pair $(\Gamma,\gamma)$ is called an {\bf oriented submanifold}
of $U$.

If $\Gamma$ is a moving $(m-1)$-dimensional boundary in $U\times I$,
an orientation on $\Gamma$ is a continuous function that assigns to each $(x,t)$ in $\Gamma$
a unit $(m-1)$-vector in $\Tan(\Gamma(t),x)$.

Now suppose that $(\Gamma,\gamma)$ is an $(m-1)$-dimensional, smooth, oriented
submanifold of $U$ and that $M\in I\MM_m(U)$.
We say that $M$ is {\bf boundary-compatible} with $(\Gamma,\gamma)$ provided
there is an $m$-dimensional, integer-multiplicity rectifiable current $A$ such that
\begin{enumerate}
\item\label{compatible1} $\partial A$ is the current given by $\Gamma$ with orientation $\gamma$ and multiplicity $1$. 
\item\label{compatible2} $M = \mu_A + 2Z$ for some $Z\in I\MM_m(U)$.
\end{enumerate}
Condition~\eqref{compatible2} is equivalent to 
\[
  \text{$\Theta(M,x) - \Theta(\mu_A,x)$ is a nonnegative even integer for $\Hh^m$ almost every $x$.}
\]
It is also equivalent to:
\[
 \text{$\mu_A\le M$, and $A$ and $M$ determine the same flat chain mod $2$.}
\]
Note that if $M$ is boundary compatible with $(\Gamma,\gamma)$, then $\partial [M]=[\Gamma]$.

The following theorem is an immediate consequence of Theorem~1.2 in \cite{white-duke}:

\begin{theorem}\label{oriented-duke-journal-theorem}
Suppose that $(\Gamma_i,\gamma_i)$ $(i=1,2,\dots)$ and $(\Gamma,\gamma)$
are smooth oriented $(m-1)$-dimensional submanifolds of $U$ such that the $(\Gamma_i,\gamma_i)$
converge weakly (i.e., as rectifiable currents) to $(\Gamma,\gamma)$.
Suppose that $M_i$ $(i=1,2,\dots)$ and $M$ are Radon measures in $I\MM_m(U)$ with the following 
properties:
\begin{enumerate}
\item\label{duke-oriented-hypothesis1} $\Var(M_i)\rightharpoonup \Var(M)$.
\item\label{duke-oriented-hypothesis2} Each $M_i$ has bounded first variation, and
\[
   q_K := \sup_i \left( \int_K |H(M_i,\cdot)|\,dM_i + \beta(M_i)(K) \right) < \infty.
\]
\item\label{duke-oriented-hypothesis3} Each $M_i$ is boundary-compatible with $(\Gamma_i,\gamma_i)$.
\end{enumerate}
Then $M$ is boundary-compatible with $(\Gamma,\gamma)$.  
\end{theorem}

\begin{definition}
Let $(\Gamma,\gamma)$ be an oriented $(m-1)$-dimensional moving boundary in $U\times I$.
An {\bf oriented Brakke flow} with boundary $(\Gamma,\gamma)$ is
a standard mean curvature flow 
\[
   t\in I\mapsto M(t)
\]
with boundary $\Gamma$ such that 
for almost every $t\in I$, $M(t)$ is boundary-compatible with $(\Gamma(t),\gamma(t))$.
\end{definition}

\begin{theorem}\label{orientation-preserved-theorem}
Suppose that 
\begin{enumerate}
\item $(\Gamma_i,\gamma_i)$ $(i=1,2,\dots)$ and $(\Gamma,\gamma)$
are moving $(m-1)$-dimensional boundaries in $U\times I$.
\item The $(\Gamma_i,\gamma_i)$
converge smoothly to $(\Gamma,\gamma)$.
\item For $i=1,2,\dots$, the map $t\in I\mapsto M_i(t)$ and is an oriented Brakke
flow with moving boundary $(\Gamma_i,\gamma_i)$.
\item $M_i(t)$ converges to $M(t)$ for all $t$.
\end{enumerate}
Then $t\mapsto M(t)$ is an oriented Brakke flow with boundary $(\Gamma,\gamma)$.
\end{theorem}

Roughly speaking, Theorem~\ref{orientation-preserved-theorem} says that the class of oriented Brakke flows
is closed under taking weak limits.

\begin{proof}
By the moving boundary version (see~\S\ref{moving-boundary-section}) of Theorem~\ref{closure-theorem},
$t\mapsto M(t)$ is a standard Brakke flow with boundary $\Gamma$.
By the moving boundary version of 
Theorem~\ref{main-flow-compactness-theorem}, for almost every $t$, there is a subsequence $M_{i(j)}(t)$
such that
\begin{equation}\label{varifolds-ok-yet-again}
      \Var(M_{i(j)}(t)) \to \Var(M(t))
\end{equation}
and
\begin{equation}\label{just-H-bound-yet-again}
    \sup_j
    \left( \int_K |H_{i(j)}(t,\cdot)|\,dM_{i(j)}(t) + \beta(M_{i(j)}(t))K \right) < \infty 
    \quad\text{for all $K\subset\subset U$}.
\end{equation}
Thus Hypotheses~\eqref{duke-oriented-hypothesis1} and~\eqref{duke-oriented-hypothesis2} 
of Theorem~\ref{oriented-duke-journal-theorem} hold for the sequence $M_{i(j)}(t)$.   
Since $M_i(t)$ is boundary-compatible with $(\Gamma_i(t),\gamma_i(t))$
for almost every $t$, Hypothesis~\eqref{duke-oriented-hypothesis3} 
of Theorem~\ref{oriented-duke-journal-theorem} holds.
Thus the conclusion holds: $M(t)$ is boundary-compatible with $(\Gamma(t),\gamma(t))$
for almost every $t$.
\end{proof}

\begin{theorem}\label{oriented-existence-theorem}
 Let $N$ be a smooth Riemannian manifold.  
If $N$ has nonempty boundary, we assume that the boundary is smooth and $m$-convex.
Let $(\Gamma,\gamma)$ be smooth, properly embedded, oriented  $(m-1)$-dimensional manifold in $N$.
Let $M_0$ be an oriented, smoothly embedded $m$-dimensional manifold in $N$ with oriented 
boundary $(\Gamma,\gamma)$.
More generally, $M_0$ can be a $m$-rectifiable set of finite $m$-dimensional Hausdorff measure
for which there is a Borel-measurable orientation of the tangent planes to $M_0$ such that the resulting
multiplicity-one current has boundary given by $(\Gamma,\gamma)$.
Then there exists an oriented Brakke flow 
\[
  t\in [0,\infty) \mapsto M(t)
\]
with boundary $(\Gamma,\gamma)$ such that 
\[
    M(t) \rightharpoonup \Hh^m\llcorner M_0
\]
as $t\to 0$, and such that $M(t)$ is boundary-compatible with $(\Gamma,\gamma)$ for almost every $t$.

If $p\notin\Gamma$ and if $M_0$ is smooth in a neighborhood of $p$, then the flow is smooth
in a spacetime neighborhood of $(p,0)$.   If  $p\in \Gamma$ and if $M_0$ is $C^{1,\alpha}$ in a neighborhood
of $p$, then the flow is parabolically $C^1$ in a spacetime neighborhood of $(p,0)$.
\end{theorem}

The proof is exactly like the proof of Theorem~\ref{existence-theorem}, except that 
one uses integral currents instead flat chains mod $2$ in the elliptic regularization construction,
and one uses Theorem~\ref{orientation-preserved-theorem} when passing to the limit.

Theorem~\ref{oriented-existence-theorem} generalizes in a straightforward way to moving
boundaries.  As in Theorem~\ref{moving-existence-theorem}, one adds the extra compactness hypothesis~\eqref{compactness-hypothesis} for the reason explained in the proof of that theorem.

\section{Appendix: A Strong Maximum Principle}

\begin{theorem}\label{maximum-principle}
Suppose that $t\in [0,T]\mapsto M(t)$ is an integral Brakke flow with moving boundary $\Gamma$
in a smooth Riemannian manifold $N$.  
Suppose that $t\in [0,T] \mapsto D(t)$ is a smooth, one-parameter family of compact, smoothly embedded, 
$m$-dimensional manifolds with boundary in $N$ that are moving by mean curvature. 
(In other words, the normal velocity at each point is equal to the mean curvature vector at that point.)
If 
\begin{gather*}
\text{$D(0)$ is disjoint from $\spt(M(0))$}, \\
\text{$\partial D(t)$ is disjoint from $\spt M(t)$ for all $t\in [0,T]$, and} \\
\text{$\Gamma(t)$ is disjoint from $D(t)$ for all $t\in [0,T]$,}
\end{gather*}
then $D(t)$ is disjoint from $\spt(M(t))$ for all $t\in [0,T]$.
\end{theorem}

\begin{proof}
Choose $\delta>0$ small enough so that 
\[
  Q:= \cup_{t\in [0,T]} \{x\in N: \dist(x,D(t)) \le \delta\}
\]
is compact, and so that for $t\in[0,T]$,
\begin{equation}\label{boundaries-not-problematic}
    \dist(p,q)>\delta 
\end{equation}
for all $p\in \Gamma(t)$ and $q\in D(t)$,  and for all $p\in \spt M(t)$ and  $q\in \partial D(t)$.

Let $\lambda< 0$ be a strict lower bound for the Ricci curvature of $N$ in the set $Q$.
We claim that 
\[
    e^{-\lambda t}\dist(D(t), \spt M(t)) > \delta
\]
for all $t\in [0,T]$.
For if not, there would be first time $t>0$ such that
\begin{equation}\label{first-dip}
   e^{-\lambda t} \dist(D(t), \spt M(t)) = \delta.
\end{equation}
At that time, there would be an arc-length parametrized geodesic
\[
    \gamma:[0,L]\to N
\]
such that
\begin{gather*}
\gamma(0) \in D(t), \\
\gamma(L) \in \spt M(t), \,\text{and} \\
   L = \dist(D(t), \spt M(t))
\end{gather*}
Note that $\gamma(0)\notin \partial D(t)$ and $\gamma(L)\notin \Gamma(t)$
by~\eqref{boundaries-not-problematic} and~\eqref{first-dip}. (Recall that $\lambda<0$).
By~\cite{hershkovits-white-avoid}*{Lemma~11} and~\cite{hershkovits-white-avoid}*{Theorem~28},
  at the point $\gamma(0)\in D(t)$, the surface
$D(\cdot)$ moves in the direction $\gamma'(0)$ faster than the mean curvature in that
direction, a contradiction.
\end{proof}

\end{document}